\documentclass[11pt,letterpaper]{amsart}
\usepackage[utf8]{inputenc}
\usepackage[english]{babel}
\usepackage[english,dvipsnames]{xcolor}
\usepackage[unicode]{hyperref}
\usepackage{verbatim}
\usepackage{amssymb}
\usepackage[foot]{amsaddr}
\hypersetup{
    pdftitle={Posets with k-outerplanar cover graphs have bounded dimension},
    pdfauthor={Maximilian Gorsky, Michał T. Seweryn},
}
\usepackage{graphicx}
\usepackage{subcaption}
\calclayout

\newtheorem{theorem}{Theorem}
\newtheorem{lemma}[theorem]{Lemma}
\newtheorem{proposition}[theorem]{Proposition}
\newtheorem*{proposition*}{Proposition}

\newtheorem{nclaim}{Claim}[theorem]
\newtheorem*{question}{Question}

\newenvironment{subproof}[1][\proofname]{%
  \begin{proof}[#1]%
}{%
  \end{proof}%
}
\DeclareMathOperator{\Inc}{Inc}
\DeclareMathOperator{\Min}{Min}
\DeclareMathOperator{\Max}{Max}
\DeclareMathOperator{\Up}{U}
\DeclareMathOperator{\Down}{D}
\DeclareMathOperator{\cover}{cover}

\newcounter{subclaimi}[theorem]

\newcommand{\Kelly}[1]{K_{#1}}

\newcommand{\Rom}[1]{\MakeUppercase{\romannumeral #1}}

\title{Posets with \(k\)-outerplanar cover graphs have bounded dimension}
\author[M.~Gorsky]{Maximilian Gorsky}
\address[M.~Gorsky]{Technische Universit{\"a}t Berlin, Germany}
\email{m.gorsky@tu-berlin.de}

\author[M.T.~Seweryn]{Michał T.\ Seweryn}
\address[M.T.~Seweryn]{Theoretical Computer Science Department, 
Faculty of Mathematics and Computer Science, Jagiellonian University, Krak\'ow, Poland}
\email{michal.seweryn@tcs.uj.edu.pl}

\date{\today}
\thanks{
Micha\l{}.T.\ Seweryn is partially supported by a Polish National Science Center grant (BEETHOVEN; UMO-2018/31/G/ST1/03718).
Maximilian Gorsky’s research is supported by the European Research Council
(ERC) under the European Union’s Horizon 2020 research and innovation programme (ERC Consolidator Grant DISTRUCT, grant agreement No 648527).
}

\begin{document}

\begin{abstract}
    In 2015, Felsner, Trotter, and Wiechert showed that posets with outerplanar cover graphs have
    bounded dimension. We generalise this result to posets with \(k\)-outerplanar cover graphs.
    Namely,
    we show that posets with \(k\)-outerplanar cover graph have dimension
    \(\mathcal{O}(k^3)\). As a consequence, we show that every poset with a planar cover graph
    and height \(h\) has dimension \(\mathcal{O}(h^3)\). This improves the previously best known bound
    of \(\mathcal{O}(h^6)\) by Kozik, Micek and Trotter.
\end{abstract}

\maketitle

\section{Introduction}

The \emph{dimension} of a poset \(P\), denoted by
\(\dim(P)\), is the least cardinality of a set
of linear orders on the ground set of \(P\)
whose intersection is the partial order of \(P\).
Dimension can be seen as a measure of poset complexity,
where the simplest posets are linearly ordered sets which
have
dimension \(1\), and the complexity grows with the dimension.

The Hasse diagram of a poset
seen as a simple undirected graph is called its cover graph.
Formally, for two elements \(x\) and \(y\) of a poset
\(P\), we say
that \(y\) \emph{covers} \(x\) when \(x < y\) in \(P\) and
there is no element \(z\) such that \(x < z < y\) in \(P\). The \emph{cover graph}
of \(P\) is a graph
whose vertex set is the ground set of \(P\) in which
two elements are adjacent when one of them covers the other
in \(P\). The cover graph of \(P\) is denoted by \(\cover(P)\).

In this paper we study finite posets with planar cover graphs,
that is posets whose cover graphs admit crossing-free drawings
on the plane (such drawings are called planar).
Note that unlike in Hasse diagrams, we do not put the
restriction  that the curve representing an edge \(x y\) with \(x < y\)
in the poset has to go upward from \(x\) to \(y\).

Intuitively, if the cover graph of a poset has a very simple
structure, then the dimension should not be too large.
There is a number of results which capture this intuition.
For instance, Trotter, and Moore showed that every poset
whose cover graph is a forest, has dimension at most \(3\)%
~\cite{trotter1977dimension}, and Felsner, Trotter, and
Viechert showed that every poset whose cover graph
is outerplanar, has dimension at most \(4\)%
~\cite{felsner2015dimension}.
(A graph is \emph{outerplanar} if it admits a
planar drawing such that all vertices
lie on the exterior face).
On the other hand, a well-known construction by
Kelly gives posets with planar cover graphs of
arbitrarily large dimension~\cite{kelly1981dimension}.
These results raise the following question:
\begin{question}\label{que:bd}
  For which minor-closed classes of graphs \(\mathcal{C}\)
  is it true that posets with cover graphs in \(\mathcal{C}\)
  have dimension bounded by a constant?
\end{question}

The classes of forests and outerplanar graphs do
have the property described in
the question and the class of planar graphs does not.
Other examples of graph classes
enjoying that property
include
graphs of pathwidth at most \(2\)%
~\cite{biro2016posets,wiechert2017}, graphs of
treewidth at most \(2\)%
~\cite{joret2017dimension,seweryn2020improved},
and graphs which exclude a fixed
\(2 \times k\) grid as a minor~\cite{huynh2020excluding}.
It was conjectured in~\cite{huynh2020excluding} that
the answer to the question above are exactly
those classes which exclude
the cover graph of a poset in Kelly's construction.

We show that for a fixed positive integer \(k\),
posets with \(k\)-outerplanar cover graphs
have bounded dimension.
A planar drawing of a graph \( G \) is called \emph{\(1\)-outerplanar} 
(or simply \emph{outerplanar})
if all vertices of \( G \) lie on the exterior face, and for every 
\(k \ge 2\), we recursively define a \emph{\(k\)-outerplanar} drawing
of \(G\) as a planar drawing of \(G\)
such that removing all vertices on the exterior
face (and the edges incident to them) yields a
\((k-1)\)-outerplanar drawing of a subgraph of \(G\).
A \(k\)-outerplanar drawing is always \(\ell\)-outerplanar for \(\ell \ge k\). 
A graph admitting a \(k\)-outerplanar drawing is called 
\emph{\( k \)-outerplanar}. See Figure~\ref{fig:3-outerplanar}.
For every \(k\), the class of \(k\)-outerplanar graphs is minor-closed
(see Lemma~\ref{lem:k-outerplanarity-minor-closed}).
\begin{figure}
    \centering
    \includegraphics{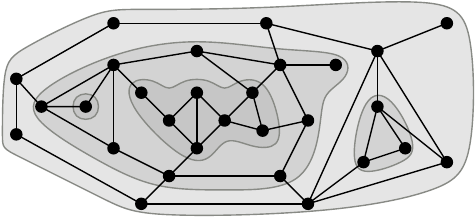}
    \caption{In this graph, after removing the vertices from the exterior face \(3\) times, we are left with no vertices. Thus the drawing and the graph are both \(3\)-outerplanar.}
    \label{fig:3-outerplanar}
\end{figure}

The main result of this paper is the following.

\begin{theorem}\label{thm:dimension-k-outerplanar-cover}
  There exists a function \(f(k) \in \mathcal{O}(k^3)\) such that
  every poset with a \(k\)-outerplanar cover graph,
  has dimension at most \(f(k)\).
\end{theorem}

Recall that a chain in a poset is a linearly ordered subset of elements of that poset and
the height of a poset is the size of a largest chain in that poset.
The following is a consequence of Theorem~\ref{thm:dimension-k-outerplanar-cover}.

\begin{theorem}\label{thm:dimension-cubic-in-height}
  There exists a function \(f(h) \in \mathcal{O}(h^3)\) such that
  every height-\(h\) poset with a planar cover graph,
  has dimension at most \(f(h)\).
\end{theorem}

Let us provide some context for this theorem.
The aforementioned construction of posets by Kelly shows that posets with planar cover graph
can have arbitrarily large dimension. However, Streib, and Trotter~\cite{streib2014dimension}
showed that the dimension of a poset with a planar cover graph is bounded in terms of its height.
That result was later generalised by Joret, Micek, and Wiechert~\cite{joret2018sparsity}, who showed that
actually for every graph class \(\mathcal{C}\) with bounded expansion, there exists a function
\(f_{\mathcal{C}}(h)\) such that \(\dim(P) \le f_{\mathcal{C}}(h)\) for any height-\(h\) poset whose
cover graph belongs to \(\mathcal{C}\). The class of planar graphs does have bounded
expansion, and their proof implied that the dimension of a height-\(h\) poset with a planar cover graph
is \(2^{\mathcal{O}(h^3)}\).
Recently, this bound was improved to a polynomial bound \(\mathcal{O}(h^6)\) by Kozik, Micek, and Trotter~\cite{kozik2019dimension}. 

Theorem~\ref{thm:dimension-k-outerplanar-cover} almost immediately improves the bound to \(\mathcal{O}(h^3)\).
Using two standard techniques called min-max reduction and
unfolding, we can transform a
poset of height \(h\) with a planar cover graph to a poset
with a \((2h-1)\)-outerplanar cover graph whose dimension
is at most two times smaller than the dimension of the original poset.
This combined with the proof of Theorem~\ref{thm:dimension-k-outerplanar-cover} implies that
the dimension of a height-\(h\) poset with a planar cover graph is
\(\mathcal{O}(h^3)\). A detailed proof is presented in Section~\ref{sec:roadmap}.

There is a lot of evidence, that the optimal bound should be \(\mathcal{O}(h)\).
All known constructions of posets with planar cover graphs do have dimension at most linear in
height.
Furthermore, in the special case of posets with planar Hasse diagrams, the dimension is
known to be \(\mathcal{O}(h)\)~\cite{joret2017planar}.
However, proving a linear bound in the general case of all posets with planar cover graphs
seems to be a challenging problem requiring a new approach.

The structure of this paper is as follows. Section~\ref{sec:prelimiaries} provides all necessary
preliminaries concerning posets and \(k\)-outerplanar graphs. In Section~\ref{sec:roadmap}, we formulate
four lemmas which combined give the proof of Theorems~\ref{thm:dimension-k-outerplanar-cover} and~\ref{thm:dimension-cubic-in-height} and we show
how these lemmas imply Theorems~\ref{thm:dimension-k-outerplanar-cover} and~\ref{thm:dimension-cubic-in-height}.
One of these four lemmas is a result from~\cite{kozik2019dimension} and the remaining three
are our contribution. In each of the Sections~\ref{sec:reduction-to-doubly-exposed},~\ref{sec:kelly-subposets}
and~\ref{sec:kelly-and-k-outerplanarity} we prove one of these three lemmas.

\section{Preliminaries}\label{sec:prelimiaries}

We use the notation \([n]\) as a shorthand for
\(\{1, \ldots, n\}\).
%We use standard notation, see%
%~\cite{diestel2017graph} for reference.
All graphs and posets in this paper are finite.

\subsection{Dimension theory}
A linear order \(L\) on the ground set of a poset \(P\) is
a \emph{linear extension} of \(P\) if \(x \le y\) in \(L\)
whenever \(x \le y\) in \(P\).
We denote by \(\Inc(P)\) the set of all ordered pairs
\((a, b)\) of incomparable elements in \(P\).
Hence the dimension of \(P\) is the least cardinality
of a nonempty family \(\mathcal{R}\)
of linear extensions of \(P\), such that
for every \((a, b) \in \Inc(P)\), the family
\(\mathcal{R}\) contains an order with \(b < a\)
(and an order with \(a < b\)).
A linear extension \(L\) of \(P\) \emph{reverses} a pair
\((a, b) \in \Inc(P)\) if \(b < a\) in \(L\) and
a subset \(I \subseteq \Inc(P)\) is \emph{reversible} in
\(P\)
if there exists a linear extension \(L\) of \(P\)
which reverses all pairs in \(I\).
The following is a useful rephrasing of the 
definition of poset dimension.

\begin{proposition}\label{prop:dim}
  If a nonempty poset \(P\) is lineary ordered, then
  \(\dim(P) = 1\). Otherwise, the
  dimension of \(P\) is the least integer
  \(d\) such that\ \(\Inc(P)\) can be partitioned into
  \(d\) sets which are reversible in \(P\).
\end{proposition}

For a poset \(P\)
and a subset  \(I \subseteq \Inc(P)\), we denote by
\(\dim_P(I)\) the least integer \(d\) such that
\(I\) can be partitioned into \(d\) sets which are
reversible in \(P\). In particular,
Proposition~\ref{prop:dim} states that
\[
\dim(P) = \dim_P(\Inc(P))\quad\textrm{if }\dim(P) \ge 2.
\]
Clearly, for \(I_1 \subseteq I_2 \subseteq \Inc(P)\),
we have \(\dim_P(I_1) \le \dim_P(I_2)\).

For a poset \( P \) and a subset \(S\) of its ground set, we define
the \emph{upset} \(\Up_P(S)\) of \(S\)
as the set of all elements \(y\) of \(P\) such that \(x \le y\) in \(P\) for some \(x \in S\) and
we define the \emph{downset} \(\Down_P(S)\) of \( S \) as 
the set of all elements \(x\) of \(P\) such that \(x \le y\) in \(P\) for some \(y \in S\).
For an element \(x \in P\), we write \(\Down_P(x)\) and \(\Up_P(x)\) as shorthands
for \(\Down_P(\{x\})\) and \(\Up_P(\{x\})\), respectively.

When \(P\) and \(Q\) are two posets such that the ground set \(X\) of
\(Q\) is a subset of the ground set of \(P\) and the orders
of the posets \(P\) and \(Q\) agree on \(X\),
we write \(Q \subseteq P\) and
we say that \(Q\) is a \emph{subposet} of \(P\) (\emph{induced by} \(X\))
and \(P\) is a \emph{superposet} of \(Q\).
When \(Q \subseteq P\), then
the restriction of any linear extension of \(P\) to \(Q\) is
a linear extension of \(Q\), so \(\dim(Q) \le \dim(P)\).
Furthermore, for every linear extension \(L\) of \(Q\) there exists a 
linear extension of \(P\) which agrees with \(L\) on the ground set of \(Q\),
and because of that, for any \(I \subseteq \Inc(Q)\) we have \(\dim_Q(I) = \dim_P(I)\).

Given two subsets \(A\) and \(B\) of the ground set of a poset
\(P\), let \(\Inc_P(A, B) = \Inc(P) \cap (A \times B)\)
and let \(\dim_P(A, B) = \dim_P(\Inc_P(A, B))\).
Let \(\Min(P)\) and \(\Max(P)\) denote respectively
the sets of minimal
and maximal elements of \(P\).
We refer to the pairs in \(\Inc_P(\Min(P), \Max(P))\) as the \emph{min-max pairs} of \(P\),
and we call \(\dim_P(\Min(P), \Max(P))\) the \emph{min-max dimension} of \(P\).

The following standard observation allows us to focus on bounding the
min-max dimension.

\begin{lemma}\label{lem:min-max-reduction}%[Min-max reduction]
  For every poset \(P\) with \(\dim(P) \ge 2\), there exists a poset \(P'\) such that
  \begin{enumerate}
  \item\label{itm:mm-ii}
  \(\cover(P')\) can be obtained from
  \(\cover(P)\) by adding some degree-\(1\) vertices,
  \item\label{itm:mm-i}
  the height of \(P'\) is equal to the height of
  \(P\), and
  \item\label{itm:mm-iv}
  \(\dim(P) \le \dim_{P'}(\Min(P'), \Max(P'))\).
  \end{enumerate}
\end{lemma}

We actually need a slightly generalised version of
Lemma~\ref{lem:min-max-reduction}.
The proof of Lemma~\ref{lem:min-max-reduction}
is a consequence of applying
the following lemma with \(A\) and \(B\) equal to the ground set of \(P\).

\begin{lemma}\label{lem:min-max-reduction-plus}%[Min-max reduction]
  Let \(P\) be a poset and let \(A\) and \(B\) be sets of
  elements in \(P\).
  Then there exist a superposet \(P'\) of \(P\) and subsets
  \(A' \subseteq \Min(P')\) and \(B' \subseteq \Max(P')\)
  such that
  \begin{enumerate}
  \item\label{itm:mmp-i}
  \(\cover(P')\) can be obtained from
  \(\cover(P)\) by adding some degree-\(1\) vertices,
  \item\label{itm:mmp-ii}
  the height of \(P'\) is equal to the height of
  \(P\),
  \item\label{itm:mmp-iii}
  \(B' \subseteq \Up_{P'}(B)\) and \(A' \subseteq \Down_{P'}(A)\), and
  \item\label{itm:mmp-iv}
  \(\dim_P(A, B) \le \dim_{P'}(A', B')\).
  \end{enumerate}
\end{lemma}

\begin{proof}
Construct a poset \(P'\) from \(P\) by adding
the following new vertices: 
for every \(a \in A \setminus \Min(P)\),
introduce a new minimal element \(a^-\) covered 
only by \(a\), and
for every \(b \in P \setminus \Max(P)\),
introduce a new maximal element \(b^+\) covering
only \(b\).
Furthermore, for each \(a \in A \cap \Min(P)\), let us
denote by \(a^-\) the element \(a\) itself, and
similarly,
for each \(b \in B \cap \Max(x)\), we let \(b^+\) denote
the element \(b\).
Let \(A' = \{a^- : a \in A\}\) and \(B' = \{b^+ : b \in B\}\).
Clearly, the poset \(P'\) satisfies~\eqref{itm:mmp-i} and~\eqref{itm:mmp-ii},
and the sets \(A'\) and \(B'\) satisfy~\eqref{itm:mmp-iii}.

It remains to show that \eqref{itm:mmp-iv} is satisfied.
Let \(d = \dim_{P'}(A', B')\) and
let \(L_1'\), \ldots, \(L_d'\) be linear extensions of \(P'\)
reversing all pairs from \(\Inc_{P'}(A', B')\).
For every pair \((a, b) \in \Inc_P(A, B)\), we have
\((a^-, b^+) \in \Inc_{P'}(A', B')\),
and if \(L_i'\) reverses \((a^-, b^+)\),
then we have \(b \le b^+ < a^- \le a\) in \(L_i'\),
so \(L_i'\) reverses  \((a, b)\) as well.
Therefore, the restrictions of
\(L_1'\), \ldots, \(L_d'\) to the ground set of \(P\)
reverse all pairs in \(\Inc_P(A, B)\).
Hence \( \dim_P(A, B) \le d\),
so the requirement of~\eqref{itm:mmp-iv} is fulfilled.
\end{proof}

Given two linear orders \( L_1 \) and \( L_2 \) on disjoint
ground sets \( S_1 \) and \( S_2 \), we write \( [ L_1 < L_2 ] \) for the linear order on \( S_1 \cup S_2 \) in which \( x < y \) is true if \( x < y \) holds in \( L_i \), for an \( i \in \{ 1,2 \} \), or if \( x \in S_1 \) and \( y \in S_2 \).
This notation extends in the intuitive way to a larger number of linear orders:
we denote by \([L_1 < \cdots < L_k]\) the order \([L_1 < [L_2 < \cdots < [L_{k-1} < L_k] \cdots]]\).

A poset is \emph{connected} if its cover graph is connected.
A \emph{component} of a poset \(P\) is a subposet of \(P\)
induced by the vertex set of a graph-theoretic component of
\(\cover(P)\).

\begin{lemma}%[Reduction to connected poset]
\label{lem:connected}
Every poset \(P\) with \(\dim(\Min(P), \Max(P)) \ge 3\) has a component \(Q\) such that
\[
  \dim_P(\Min(P), \Max(P)) = \dim_Q(\Min(Q), \Max(Q)).
\]
\end{lemma}
\begin{proof}
Let \( d = \dim_P(\Min(P), \Max(P))\) and let \( Q_1 ,\ldots , Q_k \) be the components of \( P \).
Clearly, \(\dim_{Q_i}(\Min(Q_i), \Max(Q_i)) \leq d \), for all \( i \in [k] \).
Therefore, 
it remains to show that for some \(i \in [k]\) we have \(\dim_{Q_i}(\Min(Q_i), \Max(Q_i)) \ge d\).
For the sake of contradiction, suppose that it is not the case.
Then for each \(i \in [k]\) there exist linear extensions \( L_i^1 , \ldots , L_i^{d-1}\) of \(Q_i\) reversing all pairs from \( \Inc_{Q_i}(\Min(Q_i), \Max(Q_i))\).
Note that for any \( i \), we do not require the linear extensions \( L_i^1 , \ldots , L_i^{d-1} \) to be pairwise distinct.
We define \( L^1 = [ L_k^1 < L_{k-1}^1 < \cdots < L_1^1 ] \) and for each \( j \in \{ 2, \ldots , d -1\} \), we set \( L^j = [ L_1^j < L_2^j < \cdots < L_k^j ] \).

Every pair \((a, b) \in \Inc_P(\Min(P), \Max(P))\) with \(a\) and \(b\) lying in distinct components, is reversed in \(L^1\) or \(L^2\), since \( L^1 \) and \( L^2 \) order the components of \( P \) in opposing fashion.
Furthermore, a pair \((a, b) \in \Inc_{Q_i}(\Min(Q_i), \Max(Q_i))\) which is reversed in \(L_i^j\), is reversed in \(L^j\).
Thus, the linear extensions \(L^1\), \ldots, \(L^{d-1}\) reverse all pairs from
\(\Inc_P(\Min(P), \Max(P))\), so  \( \dim_P(\Min(P), \Max(P)) \le d - 1\),
a contradiction.
\end{proof}

Let \(P\) be a connected poset and
let \(x_0 \in \Min(P)\).
Let \(A_0 = \{x_0\}\), \(B_1 = \Up_P(x_0) \cap \Max(P)\),
and for every positive integer \(i\), define inductively
\begin{align*}
  A_i = (\Down_P(B_i) \cap \Min(P)) \setminus A_{i-1},
  \quad\text{and}\quad
  B_{i+1} = (\Up_P(A_{i}) \cap \Max(P)) \setminus B_{i}.
\end{align*}
\begin{figure}[b]
    \centering
    \includegraphics{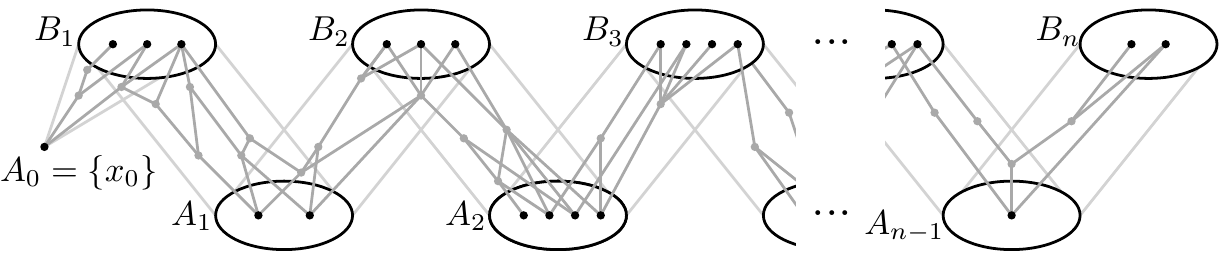}
    \caption{An example of an unfolding, with \( n \) being the greatest index such that \( A_n \cup B_n \neq \emptyset \).}
    \label{fig:unfolding}
\end{figure}
The alternating sequence of
sets \((A_0, B_1, A_1, B_2, \ldots)\) is called
the \emph{unfolding} of \(P\) from \(x_0\).
Since \(P\) is connected, the sets \(A_0\), \(A_1\), \ldots\
partition the set \(\Min(P)\), and the sets
\(B_1\), \(B_2\), \ldots\ partition \(\Max(P)\).
Note that since \( P \) is finite there are only finitely
many nonempty sets in any unfolding.
The following lemma is well-known.
\begin{lemma}\label{lem:unfolding}
  Let \(P\) be a connected poset with \(\dim(\Min(P), \Max(P)) \ge 2\)
  and let \((A_0, B_1, A_1, B_2, \ldots)\) be an unfolding
  of \(P\) from a vertex \(x_0 \in \Min(P)\).
  Then there exists a positive integer \(i\) such that
  \begin{align*}
  \dim_P(\Min(P), \Max(P)) &\le 2\cdot\dim_P(A_i, B_i) \textrm{ or}\\
  \dim_P(\Min(P), \Max(P)) &\le 2\cdot\dim_P(A_i, B_{i+1}).
  \end{align*}
\end{lemma}
\begin{proof}
Let \(n\) be the greatest integer such that \(A_n \cup B_n \neq \emptyset\).
Let \(d\) denote the maximum of the values \(\dim_P(A_i, B_i)\) and
\(\dim_P(A_i, B_{i+1})\) over all positive integers \(i\).
We need to show that \(\dim(P) \le 2d\).

For each \(i \in \{0, \ldots, n\}\), let \(X_i\) denote the subposet of \(P\) induced by
\(\Up_P(A_{i}) \setminus \Up_P(A_{i+1})\).
The ground sets of \(X_0\), \ldots, \(X_n\) partition the ground set of \(P\), and by definition of \(d\) we have
\(\dim_{X_i}(A_i, B_i) = \dim_P(A_i, B_i) \le d\) for each \(i \in \{1, \ldots, n\}\).
Let \(L^0\) be any linear extension of \(X_0\), 
and for each \( i \in \{1, \ldots, n\} \), consider \(d\) linear extensions \(L_1^i, \ldots , L_d^i\) of \( X_i \), which together reverse all pairs of \(\Inc_P(A_i, B_i)\).
For \( j \in [d] \), let \( L_j = [ L^0 < L_j^1 < \cdots L_j^n] \).
The orders \( L_1, \ldots , L_d \) are linear extensions of \(P\) which together reverse all pairs \( (a,b) \in \Inc_P(\Min(P), \Max(P))\) with \( a \in A_i \) and \( b \in B_j \) for \( j \le i \).

To reverse the remaining pairs, for each \(i \in \{1, \ldots, n\}\), let \( Y_i \) denote the subposet of \(P\) induced by \( \Down_P(B_i) \setminus \Down_P(B_{i+1}) \) in \( P \).
The ground sets of \(Y_1\), \ldots, \(Y_n\) partition the ground set of \(P\) and by definition of \(d\)
we have \(\dim_{Y_{i+1}}(A_{i}, B_{i+1}) = \dim_P(A_{i}, B_{i+1}) \le d\).
Let \(M^1\) be any linear extension of \(Y^1\) and for each
\(i \in \{2, \ldots, n\}\), consider \(d\) linear extensions \(M_1^i, \ldots , M_d^i\) of \( Y_i \), which together reverse all pairs of \(\Inc_P(A_{i-1},B_{i})\).
For \( j \in [d] \), let \( M_j = [ M_j^{k-1} < \cdots < M_j^2 < M^1] \).
The orders \( M_1, \ldots , M_d \) are linear extensions of \(P\) which together reverse all pairs \( (a, b) \in \Inc_P(\Min(P), \Max(P))\) with \( a \in A_i \) and \( b \in B_j \) for \( j \geq i + 1 \).
As a consequence, the linear extensions \( L_1, \ldots , L_d, M_1, \ldots , M_d \) of \(P\)
together reverse all pairs in
\(\Inc(\Min(P), \Max(P))\) and thus \(\dim_P(\Min(P), \Max(P)) \leq 2d\).
\end{proof}

A subposet \(Q\) of a poset \(P\) is \emph{convex}
if, whenever \(x < y < z\) in \(P\) and both \(x\) and \(z\) belong to \(Q\),
\(y\) belongs to \(Q\) too. If \(Q\) is a convex subposet of \(P\),
then the cover graph of \(Q\) is an induced subgraph of
\(P\). %and every drawing \(\mathbb{D}\) of the cover graph of 
%\(P\) contains a drawing of the cover graph of \(Q\).
\iffalse
The \emph{dual} of a poset \(P\) is a poset \(P^d\) with the same ground set as \(P\)
but in which \(x \le y\) if \(y \le x\) in \(P\).
Clearly, \(\dim_P(\Min(P), \Max(P)) = \dim_{P^d}(\Min(P^d), \Max(P^d))\) and more generally, for
\(I \subseteq \Inc(\Min(P), \Max(P))\), if \(I^{-1} = \{(b, a) : (a, b) \in I\}\),
then \(\dim_P(I) = \dim_{P^d}(I^{-1})\), and \(\dim_P(A, B) = \dim_{P^d}(B, A)\) for \(A \subseteq \Min(P)\) and
\(B \subseteq \Max(P)\).
\fi

The notion of unfolding allows us to reduce a connected poset \(P\) to a convex subposet of the form
\(\Up_P(A_i) \cap \Down_P(B_i)\) or \(\Up_P(A_i) \cap \Down_P(B_{i+1})\) whose min-max dimension
is at most \(2\) times smaller. The key properties of such a subposet \(Q\) are captured by the following lemma.
\begin{lemma}\label{lem:unfolding-SQ}
  Let \(P\) be a connected poset with \(\dim_P(\Min(P), \Max(P)) \ge 2\) and let \(x_0 \in \Min(P)\). Then there exist convex subposets
  \(S\) and \(Q\) of \(P\) with \(\Min(Q) \subseteq \Min(P)\) and \(\Max(Q) \subseteq \Max(P)\) such that
  \begin{enumerate}
      \item\label{itm:S-comp} \(S\) is a component of \(P - Q\) containing \(x_0\),
      \item\label{itm:dim-P-Q} \(\dim_P(\Min(P), \Max(P)) \le 2 \cdot \dim_Q(\Min(Q), \Max(Q)))\), and
      \item\label{itm:either} either \(\Max(Q) \subseteq \Up_P(S)\) and \(\Down_P(S) \cap Q = \emptyset\), or
        \(\Min(Q) \subseteq \Down_P(S)\) and \(\Up_P(S) \cap Q = \emptyset\).
  \end{enumerate}
\end{lemma}
\begin{proof}
  Let \(d = \dim_P(\Min(P), \Max(P))\)
  and let \((A_0, B_1, A_1, B_2, \ldots)\) be the unfolding of \(P\) from \(x_0\).
  By Lemma~\ref{lem:unfolding}, there exists a positive integer \(i\)
  such that either \(d \le 2 \cdot \dim_P(A_i, B_i)\), or
  \(d \le 2 \cdot \dim_P(A_i, B_{i+1})\).
  In the former case, let \(Q = \Up_P(A_i) \cap \Down_P(B_i)\)
  and in the latter case, let \(Q = \Up_P(A_i) \cap \Down_P(B_{i+1})\).
  We have \(\dim_P(\Min(P), \Max(P)) = d \le 2 \cdot \dim_Q(\Min(Q), \Max(Q))\),
  so~\eqref{itm:dim-P-Q} holds. It is easy to verify that
  \(S = \Down_P(B_1 \cup \cdots \cup B_i) \setminus Q\)
  satisfies~\eqref{itm:S-comp} and~\eqref{itm:either}.
\end{proof}

If \(x_0\), \ldots, \(x_k\) are elements of a poset \(P\) such that
\(x_i\) covers \(x_{i-1}\) for each \(i \in [k]\),
then these elements induce a path in \(\cover(P)\), which we call a
\emph{witnessing path} (\emph{from \(x_0\) to \(x_k\)}) in \(P\).
Since the cover graph can be derived from a Hasse diagram,
for any pair of elements \(x\) and \(y\) with \(x \le y\) in \(P\)
there exists a witnessing path from \(x\) to \(y\) in \(P\).

The following is a variant of Lemma~\ref{lem:unfolding-SQ} suited for posets with planar cover graphs.

\begin{lemma}\label{lem:unfolding-planar}
  Let \(P\) be a poset with \(\dim_P(\Min(P), \Max(P)) \ge 3\) and
  with a fixed planar drawing of its cover graph.
  Then there exists a convex subposet \(Q\) of \(P\) with
  \[\dim_P(\Min(P), \Max(P)) \le 2 \cdot \dim_Q(\Min(Q), \Max(Q))\]
  such that if \(V_1\) denotes the set of vertices lying on the exterior face in the
  induced drawing of the cover graph of \(Q\), then
  either \(\Max(Q) \subseteq \Up_Q(V_1)\), or \(\Min(Q) \subseteq \Down_Q(V_1)\).
\end{lemma}
\begin{proof}
By Lemma~\ref{lem:connected},
  after replacing \(P\) with its component of the same min-max dimension
  we may assume that \(P\) is connected.
  For any edge lying on the exterior face, at most one of its ends is a
  minimal element of \(P\), so in particular there exists a non-minimal
  element of \(P\) on the exterior face.
  Let \(z\) be such a vertex.
  Let \(P + z^-\) be a poset obtained from \(P\) by adding a minimal element \(z^{-}\)
  covered only by \(z\) and extend the drawing to a planar drawing of \(\cover(P+z^-)\) with
  \(z^-\) on the exterior face of the cover graph.
  Apply Lemma~\ref{lem:unfolding-SQ} to \(P + z^{-}\) and \(z^-\),
  and let \(S\) and \(Q\) be the resulting convex subposets.
  Let us assume \(\Max(Q) \subseteq \Up_{P+z^-}(S)\) and
  \(\Down_{P+z^-}(S) \cap Q = \emptyset\).
  For every \(b \in \Max(Q)\) there exists a witnessing path from an element of \(S\) to \(b\) in \(P\).
  The least element belonging to \(Q\) on such a path lies on the exterior face of
  \(\cover(Q)\) because
  \(S\) is a component of \((P+z^-) - Q\) containing the vertex \(z^-\) lying on the exterior face of
  the drawing.
  Hence, every element of \(\Max(Q)\)
  is comparable with a vertex on the exterior face of the induced drawing of \(\cover(Q)\).
  By a symmetric argument, if \(\Min(Q) \subseteq \Down_{P+z^-}(S)\) and \(\Up_{P+z^-}(S) \cap Q = \emptyset\),
  then every element of \(\Min(Q)\) is comparable with a vertex on the exterior face of the induced drawing
  of \(\cover(Q)\).
  Therefore \(Q\) satisfies the lemma.
\end{proof}

For a positive integer \(n\), the \emph{standard example}
\(S_n\) is a poset consisting of \(2n\) elements
\(a_1\), \ldots, \(a_n\), \(b_1\), \ldots, \(b_n\), where
\(x < y\) in \(S_n\) if and only if \(x = a_i\) and
\(y = b_j\) for some \(i, j \in [n]\) with \( i \neq j \).
The standard example \(S_n\) is a canonical example of a poset
of dimension \(n\).
In a poset \(P\), every subposet isomorphic to \(S_n\)
can be identified with a set of the form
\(\{(a_1, b_1), \ldots, (a_n, b_n)\}\). 
Therefore, by a \emph{standard example} (in \(P\))
we also mean a subset of the form
\(\{(a_1, b_1), \ldots, (a_n, b_n)\} \subseteq \Inc(P)\)
where \(a_i < b_j\) in \(P\) for \(i \neq j\).
\begin{figure}
    \centering
    \includegraphics{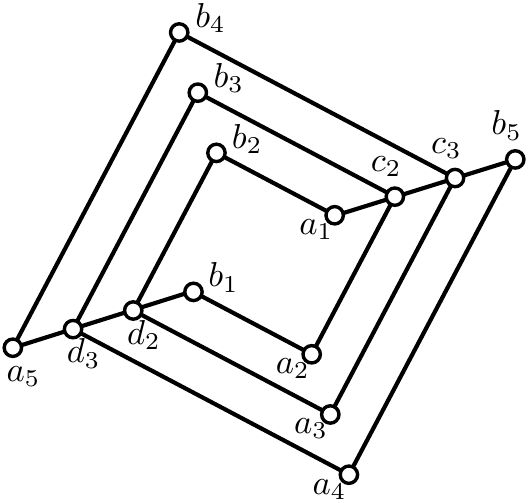}
    \caption{The Kelly poset \( \Kelly{5} \).}
    \label{fig:kelly}
\end{figure}

Kelly~\cite{kelly1981dimension} observed that for every \(n \ge 3\),
the power set of \([n]\) ordered by inclusion contains 
a subposet \(\Kelly{n}\) which has a planar cover graph
containing a standard example of size \(n\). 
The elements of the \emph{Kelly poset} \(\Kelly{n}\) are the following subsets of \([n]\):
\begin{alignat*}{3}
    a_i &= \{i\}&&\textrm{for }i \in [n],\\
    b_i &= [n] \setminus \{i\}&&\textrm{for }i \in [n],\\
    c_i &= \{1, \ldots, i\}&&\textrm{for }i \in [n-1],\\
    d_i &= \{i+1, \ldots, n\}\quad&&\textrm{for }i \in [n-1].
\end{alignat*}
Note that \(c_1 = a_1\), \(c_{n-1} = b_n\), \(d_1 = b_1\) and \(d_{n-1} = a_n\).
See Figure~\ref{fig:kelly}.

\subsection{\(k\)-outerplanarity}
For a planar drawing of a graph \(G\) we define its \emph{layering} as a
sequence of sets \((V_1, V_2, \ldots)\)
where \(V_1\) is the set of vertices lying on the exterior face of the drawing,
and for \(i \ge 2\), the set \(V_i\) is the set of vertices lying on the exterior face
in the induced drawing of \(G - (V_1 \cup \ldots \cup V_{i-1})\).
Starting from some point
all sets in the sequence \(V_1\), \(V_2\), \ldots are empty.
Each set \(V_i\) is called a \emph{layer} and the layers of any planar drawing of \(G\)
partition the set \(V(G)\).
Thus, the drawing is \(k\)-outerplanar if and only
if \(V_i = \emptyset\) for all \(i \ge k + 1\).
Every edge of \(G\) either has both ends in a set \(V_i\), or has ends in two consecutive sets
\(V_i\) and \(V_{i+1}\). For completeness we include the proofs of several easy and useful observations
about \(k\)-outerplanarity.

\begin{lemma}\label{lem:k-outerplanar-subgraph}
  If \(G\) is a graph with a fixed \(k\)-outerplanar drawing, then 
  the induced drawing of every subgraph of \(G\) is \(k\)-outerplanar as well.
\end{lemma}
\begin{proof}
  Let \((V_1, V_2, \ldots)\) be the layering of the drawing of \(G\) and let
  \((V_1', V_2', \ldots)\) be the layering of the induced drawing of a subgraph
  \(H\) of \(G\). For each \(i \ge 1\), let \(G_i = G[\bigcup_{j \ge i} V_j]\)
  and \(H_i = H[\bigcup_{j \ge i} V_j']\).
  We prove by induction that \(H_i \subseteq G_i\) for every \(i \ge 1\).
  This holds when \(i = 1\), as \(V(H_1) = V(H) \subseteq V(G) = V(G_1)\).
  For the induction step, let \(i \ge 2\), and suppose that \(H_{i-1} \subseteq G_{i-1}\).
  Every vertex of \(H_{i-1}\) which lies on the exterior face of \(G_{i-1}\), must
  lie on the exterior face of \(H_{i-1}\) as well.
  Thus \(V(H_{i-1}) \cap V_{i-1} \subseteq V_{i-1}'\) and as a consequence
  \[
    H_i = H_{i-1} - V_{i-1}' \subseteq G_{i-1} - V_{i-1} = G_i
  \]
  as claimed. The inductive proof follows, and therefore \(H_{k+1} \subseteq G_{k+1} = \emptyset\).
  Therefore the induced drawing of \(H\) is \(k\)-outerplanar.
\end{proof}

\begin{lemma}\label{lem:k-outerplanarity-minor-closed}
  A minor of a \(k\)-outerplanar graph is \(k\)-outerplanar.
\end{lemma}
\begin{proof}
  Every minor of a graph can be obtained from one of its subgraphs by a sequence of
  edge contractions. Therefore, by Lemma~\ref{lem:k-outerplanar-subgraph},
  it suffices to show that if \(G\) is a \(k\)-outerplanar graph and \(e = uv \in E(G)\),
  then the graph \(G/e\) obtained by contracting the edge \(e\) to a vertex \(w\)
  is \(k\)-outerplanar. Fix a \(k\)-outerplanar drawing of \(G\) with a layering
  \((V_1, V_2, \ldots)\).
  After deleting the vertices \(u\) and \(v\) and all edges incident to them
  from the drawing of \(G\), there emerges a face containing all vertices
  which are adjacent to \(w\) in \(G/e\).
  Draw the vertex \(w\) in the interior of the region bounded by that face
  and draw edges connecting \(w\) to its neighbours in \(G/e\) to obtain a planar
  drawing of \(G/e\). Denote the layering of that drawing by  \((V_1', V_2', \ldots)\).
  Let \(i\) denote the least integer such that \(e\) has an end in \(V_i\).
  We have \(V_j' = V_j\) for \(j < i\), \(V_i' = (V_i \setminus \{u, v\}) \cup \{w\}\),
  and \(\bigcup_{j > i} V_j' = (\bigcup_{j > i} V_j) \setminus \{u, v\}\).
  Hence, by Lemma~\ref{lem:k-outerplanar-subgraph}, the induced drawing of
  \(G[\bigcup_{j > i} V_j']\) is \((k-i)\)-outerplanar, and therefore the drawing
  of \(G/e\) is \(k\)-outerplanar.
\end{proof}

\begin{lemma}
  A graph obtained from a \(k\)-outerplanar graph by adding some degree-\(1\) vertices is
  \(k\)-outerplanar.
\end{lemma}
\begin{proof}
  Let \(G\) be a graph with a fixed \(k\)-outerplanar drawing,
  and let \((V_1, V_2, \ldots)\) be the layering of the drawing.
  Consider a vertex \(v\) of \(G\) which belongs to a layer \(V_i\).
  When we want to add a new vertex \(u\) attached to \(v\), we can
  extend the planar drawing so that the vertex \(u\) is
  drawn in the unbounded region of \(G[V_i \cup V_{i+1} \cup \ldots]\).
  The layering of the resulting drawing is the same as the original one except that
  the vertex \(u\) is added to the layer \(V_i\).
  Therefore adding a single degree-\(1\) vertex preserves \(k\)-outerplanarity of
  \(G\). As we can add any number of degree-\(1\) vertices one at a time, the
  lemma follows.
\end{proof}

\section{The roadmap}\label{sec:roadmap}

The proof of Theorem~\ref{thm:dimension-k-outerplanar-cover}
is a consequence of four lemmas.
Lemmas~\ref{lem:doubly-exposed-k-outerplanar},%
~\ref{lem:se-vs-kelly} and \ref{lem:outerplanarity-of-kelly}
are our contribution, while
Lemma~\ref{lem:kmt} is a result by Kozik, Micek, and Trotter.
In this section we give the statement of these four lemmas, and
we show how they imply the Theorems~\ref{thm:dimension-k-outerplanar-cover}
and~\ref{thm:dimension-cubic-in-height}.

Let \(P\) be a poset with a planar cover graph and let \(I \subseteq \Inc(P)\).
Following the terminology from \cite{kozik2019dimension}, we say that
the set \(I\) is
\emph{doubly exposed} if there exist
a planar drawing of
the cover graph of \(P\)
and two vertices \(x_0\) and \(y_0\) on the exterior face such that
for every \((a, b) \in I\) we have \(x_0 \le b\) and \(a \le y_0\)
in \(P\).
In such a setting we say that \(I\) is doubly exposed
\emph{by} the pair \((x_0, y_0)\).

\begin{lemma}\label{lem:doubly-exposed-k-outerplanar}
  Let \(k \ge 1\) and 
  let \(P\) be a poset with a \(k\)-outerplanar cover graph
  such that \(\dim(P) \ge 3\).
  Then there exist a poset \(P'\) with
  a \((k+1)\)-outerplanar cover graph
  and a doubly exposed set \(I \subseteq \Inc_{P'}(\Min(P'), \Max(P'))\)
  such that
  \[\dim(P) \le 4k \cdot \dim_{P'}(I).\]
\end{lemma}

For a poset \(P\) and a nonempty subset \(I \subseteq \Inc(\Min(P), \Max(P))\), let
\(\rho_P(I)\) denote the size of a largest standard example in
\(I\).
The following lemma is proven by Kozik, Micek, and Trotter in
\cite{kozik2019dimension}.

\begin{lemma}[\cite{kozik2019dimension}]\label{lem:kmt}
  If \(P\) is a poset with a planar cover graph and
  \(I\) is a doubly exposed subset of \(\Inc_P(\Min(P), \Max(P))\) in \(P\), then
  \[
    \dim_P(I) \le \rho_P(I)^2.
  \]
\end{lemma}

In~\cite{kozik2019dimension}, it is shown that
if \(P\) is a poset with a planar cover graph and
\(I \subseteq \Inc(\Min(P), \Max(P))\) is a doubly exposed set in \(P\), then
\(P\) contains a chain of size \(\Omega(\rho_P(I))\).
We generalise this result by showing that in such a setting \(P\) actually contains a subposet isomorphic to a Kelly poset \(\Kelly{n}\) with \(n = \Omega(\rho_P(I))\).

For a poset \(P\), let \(\kappa(P)\) denote the largest integer \(n \ge 3\)
such that \(P\) contains a subposet isomorphic to the Kelly poset \(\Kelly{n}\).
(If no such \(n\) exists, set \(\kappa(P) = 2\).)

\begin{lemma}\label{lem:se-vs-kelly}
  If \(P\) is a poset with a planar cover graph and
  \(I \subseteq \Inc_P(\Min(P), \Max(P))\) is a doubly exposed
  set in \(P\), then
  \[\rho_P(I) \le 360 \cdot (\kappa(P) + 1).\]
\end{lemma}

Finally, we show that \(\kappa(P)\) is at most linear in \(k\) for a  poset \(P\)
with a \(k\)-outerplanar cover graph.   

\begin{lemma}\label{lem:outerplanarity-of-kelly}
  For every \(k \ge 1\), 
  if \(P\) is a poset with a \(k\)-outerplanar cover graph,
  then
  \[\kappa(P) \le 4k + 2.\]
\end{lemma}

The proof of Theorem~\ref{thm:dimension-k-outerplanar-cover} can now be obtained
as a composition of the Lemmas~\ref{lem:doubly-exposed-k-outerplanar},%
~\ref{lem:kmt},~\ref{lem:se-vs-kelly}, and~\ref{lem:outerplanarity-of-kelly}.

\begin{proof}[Proof of Theorem~\ref{thm:dimension-k-outerplanar-cover}]
Let \(k\) be a positive integer and
let \(P\) be a poset with a \(k\)-outerplanar cover graph such that \(\dim(P) \ge 3\).
Let \(P'\) and \(I\) be obtained by
applying Lemma~\ref{lem:doubly-exposed-k-outerplanar} to the poset \(P\).
This way \(P'\) is a poset with a \((k+1)\)-outerplanar graph and
\(I\) is a doubly exposed set of min-max pairs in \(P'\) such that
\begin{align*}
  \dim(P) &\le 4k \cdot \dim_{P'}(I)\\
  &\le 4k \cdot \rho_{P'}(I)^2 &&\textrm{(by Lemma~\ref{lem:kmt})}\\
  &\le 4k \cdot (360 \cdot (\kappa(P') + 1))^2 &&\textrm{(by Lemma~\ref{lem:se-vs-kelly})}\\
  &\le 4k \cdot (360 \cdot (4(k+1) + 2 + 1))^2 &&
  \textrm{(by Lemma~\ref{lem:outerplanarity-of-kelly})}.
\end{align*}
  Hence the theorem holds for the function \(f(k) = 4k \cdot (360 \cdot (4k+7))^2\).
\end{proof}

As a consequence of Theorem~\ref{thm:dimension-k-outerplanar-cover}, we can easily prove that the dimension of a height-\(h\) poset with a planar cover graph is \( \mathcal{O}(h^3) \).

\begin{proof}[Proof of Theorem~\ref{thm:dimension-cubic-in-height}]
Let \(P\) be a poset of height \(h\) with a planar cover graph.
By Lemma~\ref{lem:min-max-reduction}, there exists a poset \(P'\)
of height \(h\) with a planar cover graph such that
\(\dim(P) \le \dim_{P'}(\Min(P'), \Max(P'))\).
Take any such \(P'\) and fix a planar drawing of the cover graph of \(P'\).
By Lemma~\ref{lem:unfolding-planar}, there exists a convex subposet \(Q\) of
\(P'\) such that
\[
  \dim_{P'}(\Min(P'), \Max(P')) \le 2 \cdot \dim_{Q}(\Min(Q), \Max(Q)) \le 2 \cdot \dim(Q).
\]
and either every element of \(\Max(Q)\) or every element of
\(\Max(Q)\) is comparable in \(Q\) with an element lying on the exterior face of
\(\cover(Q)\). Without loss of generality we assume that every element of \(\Max(Q)\) is comparable with an
element lying on the exterior face of \(Q\).

Let \((V_1, V_2, \ldots)\) be the layering of the induced drawing of \(\cover(Q)\).
Recall that every edge of \(\cover(Q)\) either has two ends in one layer, or is between two consecutive
layers. Since the height of \(Q\) is at most \(h\), every witnessing path has at most \(h - 1\) edges,
and therefore for any pair of comparable elements \(x\) and \(y\), if \(x \in V_i\) and \(y \in V_j\), then
\(|i - j| \le h - 1\).
By our assumption, every maximal element of \(Q\) is comparable with an element of \(V_1\) and thus
\(\Max(Q) \subseteq V_1 \cup \cdots \cup V_h\). But every element of \(Q\) is comparable with some maximal
element, so all elements of \(Q\) are in \(V_1 \cup \cdots \cup V_{2h-1}\).
Hence the induced drawing of \(\cover(Q)\) is \((2h-1)\)-outerplanar and therefore
\begin{align*}
  \dim(P) &\le \dim_{P'}(\Min(P'), \Max(P')) \le 2\cdot\dim(Q) \le
  2 \cdot f(2h-1) \in \mathcal{O}(h^3)
\end{align*}
where \(f(k)\) is a function satisfying Theorem~\ref{thm:dimension-k-outerplanar-cover}.
\end{proof}

\section{Reduction to doubly exposed posets}\label{sec:reduction-to-doubly-exposed}

Ahead of the proof of Lemma~\ref{lem:doubly-exposed-k-outerplanar}, let us first present an elementary lemma about poset dimension.
\begin{lemma}\label{lem:dimAB}
  Let \(P\) be a poset and let \(A\) and \(B\) be subsets of its ground set such that
  \(\dim_P(A, B) \ge 1\).
  Then
  \[
    \dim_P(A, B) = \dim_P(A,  B \cap \Up_P(A))
  \]
\end{lemma}
\begin{proof}
  The inequality \(\dim_P(A, B) \ge \dim_P(A,  B \cap \Up_P(A))\) is trivial.
  For the proof of the other inequality, let \(Q\) denote the subposet of \(P\) induced by \(\Up_P(A)\),
  let \(d = \dim_P(A,  B \cap \Up_P(A)) = \dim_Q(A, B \cap \Up_P(A))\),
  and let \(L_1\), \ldots, \(L_d\) be linear extensions of \(Q\) which
  together reverse all pairs in \(\Inc_P(A,  B \cap \Up_P(A))\).
  Let \(L^0\) be a linear extension of \(P - \Up_P(A)\).
  Now the linear orders \([L^0 < L_1]\), \ldots, \([L^0 < L_d]\)
  are linear extensions of \(P\) which reverse all pairs in \(\Inc_P(A, B)\), so
  \(\dim_P(A, B) \le d\). The lemma follows.
\end{proof}

\begin{proof}[Proof of Lemma~\ref{lem:doubly-exposed-k-outerplanar}]

Let \( P \) be a poset with a \( k \)-outerplanar cover graph.
When \(\dim(P) \le 4k\), we can take \(Q = S_2\), \(x_0 = a_1\), \(y_0 = b_1\) and \(I = \{(a_2, b_2)\}\),
so that \(\cover(Q)\) is \(1\)-outerplanar and \(\dim(P) \le 4k = 4k \cdot \dim_Q(I)\).
Therefore we assume that
\[ \dim(P) > 4k.\]
Addition of degree-\(1\) vertices to a \(k\)-outerplanar graph preserves its \(k\)-outerplanarity, so
by Lemma~\ref{lem:min-max-reduction} we may assume that
\[
  \dim(P) = \dim(\Min(P), \Max(P)).
\]
and by Lemma~\ref{lem:connected} we may assume that \(P\)
is connected.

Let us fix a \(k\)-outerplanar drawing of \(\cover(P)\).
Let \(Q\) be the convex subposet of \(P\)
obtained by applying Lemma~\ref{lem:unfolding-planar}.
In particular, we have
\[
  \dim(\Min(P), \Max(P)) \le 2 \cdot \dim_Q(\Min(Q), \Max(Q)).
\]
Denote the layering of the induced \(k\)-outerplanar drawing
of \(\cover(Q)\) by \((V_1, V_2, \ldots)\).
We have \(\Max(Q) \subseteq \Up_Q(V_1)\) or \(\Min(Q) \subseteq \Down_Q(V_1)\).
Without loss of generality, we assume that \(\Max(Q) \subseteq \Up_Q(V_1)\).
Since the drawing of \(\cover(Q)\) is \(k\)-outerplanar,
we have \(V_i = \emptyset\) for \(i \ge k + 1\).

For every \(a \in \Min(Q)\), let \(\alpha(a)\) denote the least \(i \in [k]\)
such that \(\Up_Q(a) \cap V_i \neq \emptyset\), and
for each \(i \in [k]\), let \(A_i\) denote the set of all \(a \in \Min(Q)\),
such that \(\alpha(a) = i\).
The sets \(A_1\), \ldots, \(A_k\) partition \(\Min(Q)\) and therefore
\[
  \dim_Q(\Min(Q), \Max(Q)) \le
  \sum_{i=1}^k \dim_Q(A_i, \Max(Q)). 
\]
For each \(i \in [k]\), let \(Q_i\) denote the subposet of \(Q\)
induced by \(\Up_Q(A_i)\).
We have \(\Min(Q_i) = A_i\) and \(\Max(Q_i) = \Max(Q) \cap \Up_P(A_i)\),
so by Lemma~\ref{lem:dimAB} for each \(i \in [k]\) we have
\[
  \dim_Q(A_i, \Max(Q)) = \dim_{Q_i}(\Min(Q_i), \Max(Q_i))
\]

Consider any of the posets \(Q_i\).
By definition of \(A_i\), all elements of \(Q_i\) belong to
\(V_i \cup \cdots \cup V_k\) and we have \(A_i = \Min(Q_i) \subseteq \Down_Q(V_i)\).
Any witnessing path from \(a\) to an element of \(V_i\)
has to contain a vertex lying on the exterior face of \(\cover(Q_i)\).
Hence, every element of \(\Min(Q_i)\) is comparable with a vertex lying
on the exterior face of \(\cover(Q_i)\).
Furthermore, we have \(\Max(Q_i) \subseteq \Up_Q(V_1)\),
so a similar argument shows that every element of
\(\Max(Q_i)\) is comparable with a vertex lying
on the exterior face of \(\cover(Q_i)\).

Let \(Q'\) denote a poset among \(Q_1\), \ldots, \(Q_k\)
for which \(\dim_{Q'}(\Min(Q'), \Max(Q'))\)
is largest. 
Summarizing, we have
\begin{align*}
\dim(P) &= \dim_P(\Min(P), \Max(P))\\
&\le 2 \cdot \dim_Q(\Min(Q), \Max(Q))\\
&\le 2 \cdot \sum_{i=1}^k \dim_Q(A_i, \Max(Q))\\
&= 2 \cdot \sum_{i=1}^k \dim_{Q_i}(\Min(Q_i), \Max(Q_i))\\
&\le 2k \cdot \dim_{Q'}(\Min(Q'), \Max(Q')).    
\end{align*}

By our assumption, we have \(\dim(P) > 4k\), so
\(\dim_{Q'}(\Min(Q'), \Max(Q')) \ge 3\).
By Lemma~\ref{lem:connected}, there exists a component
\(Q''\) of \(Q'\) such that
\[\dim_{Q'}(\Min(Q'), \Max(Q')) = \dim_{Q''}(\Min(Q''), \Max(Q'')).\]
Let us fix such a component \(Q''\), and
let \(V_1''\) denote the set of vertices on
the exterior face of \(\cover(Q'')\).
Note that the drawing of \(\cover(Q')\) is \(k\)-outerplanar,
and we have \(\Max(Q'') \subseteq \Up_{Q''}(V_1)\) and \(\Min(Q'') \subseteq \Down_{Q''}(V_1'')\).

We may assume that \(Q''\) contains a minimal element on the
exterior face of its cover graph.
If this is not the case, we may simply introduce a new miminal element
covered by a single vertex which lies on the exterior face of \(\cover(Q')\).
We apply Lemma~\ref{lem:unfolding-SQ} to \(Q''\) and a minimal element on the exterior face
to obtain convex subposets
\(S\) and \(R\) of \(Q''\) with \(\Min(R) \subseteq \Min(Q'')\) and \(\Max(R) \subseteq \Max(Q'')\)
such that
\[
  \dim(P) \le 2k \cdot \dim_{Q''}(\Min(Q''), \Max(Q'')) \le 4k \cdot \dim_{R}(\Min(R), \Max(R)),
\]
\(S\) is a component of \(Q'' - R\) containing a vertex from the exterior face of \(\cover(Q'')\), and either
\(\Max(R) \subseteq \Up_{Q''}(S)\) and \(\Down_{Q''}(S) \cap R = \emptyset\),
or \(\Min(R) \subseteq \Down_{Q''}(S)\) and \(\Up_{Q''}(S) \cap R = \emptyset\).
Without loss of generality, we assume that the former holds, that is:
\[
\Max(R) \subseteq \Up_{Q''}(S)
\quad\textrm{and}\quad
\Down_{Q''}(S) \cap R = \emptyset.
\]

We add two elements to the poset \(R\): a minimal
element \( x_0 \) below all elements of \(R\) which cover an element of \(S\),
and a maximal element \(y_0\) above all vertices on the exterior face of \(\cover(Q'')\)
which do not belong to \(S\).
Formally, let \(R'\) be a post obtained from \(R\) by adding a minimal element \(x_0\) and a maximal element \(y_0\) such that in \(R'\), for every element \(z\) we have
\begin{align*}
&x_0 < z \textrm{ if and only if }z \in (\Up_{Q''}(S) \cap R) \cup \{y_0\},\textrm{ and }\\
&z < y_0 \textrm{ if and only if }z \in (\Down_{Q''}(V_1'') \cap R) \cup \{x_0\}.
\end{align*}
In the poset \(R'\), the only elements which can cover \(x_0\) are \(y_0\) and the vertices adjacent to \(S\)
in \(\cover(Q'')\),
whereas the only elements which can be covered by  \(y_0\) are \(x_0\) and the vertices from the exterior face
of \(\cover(R)\) which do not belong to \(S\).
Hence the induced drawing of \(\cover(R)\) can be extended to a drawing
of \(\cover(R')\) with \(x_0\) and \(y_0\) on the exterior face.
Such a drawing witnesses that \(\cover(R')\) is \((k+1)\)-outerplanar
(after removing the vertices from the exterior face we are left with an induced \(k\)-outerplanar drawing
of a subgraph of \(\cover(R)\)).

The set \(\Inc_{R'}(\Min(R), \Max(R))\) is doubly exposed by \((x_0, y_0)\) in \(R'\):
for every \((a, b) \in I\), we have
\[
  b \in \Max(R) \subseteq \Up_{Q'}(S) \cap R \subseteq \Up_{Q''}(x_0), \text{ and} \\
\]
\[
    a \in \Min(R) \subseteq \Down_{Q'}(\Max(R)) \setminus S \subseteq \Down_{Q''}(V' \setminus S) \subseteq \Down_{Q''}(y_0).
\]
Furthermore, we have \(\dim(P) \le 4k \cdot \dim_{R'}(\Min(R), \Max(R))\).
The only thing which prevents \(R'\) and \(\Inc_{R'}(\Min(R), \Max(R))\)
from satisfying the lemma is the requirement that
\(\Inc_{R'}(\Min(R), \Max(R))\) should  be a set of min-max pairs in \(R'\):
some maximal elements of \(R\) may be covered by
\(y_0\) in \(Q''\). We overcome this, by applying Lemma~\ref{lem:min-max-reduction-plus}
to \(R'\) and the sets \(A = \Min(R)\) and \(B = \Max(R)\).
Let \(P'\), \(A'\) and \(B'\) be the resulting poset and sets,
so that \(\dim_{Q''}(\Min(R), \Max(R)) \le \dim_{P'}(A', B')\).
Let \(I = \Inc_{P'}(A', B') \subseteq \Inc_{P'}(\Min(P'), \Max(P'))\).
This way,  by the item~\eqref{itm:mmp-iii}
of Lemma~\ref{lem:min-max-reduction-plus}, the set \(I\) is doubly exposed by \((x_0, y_0)\) in \(P'\) , and 
the cover graph of \(P'\) is \((k+1)\)-outerplanar as it is obtained by adding degree-\(1\) vertices to \(\cover(R')\).
The proof of Lemma~\ref{lem:doubly-exposed-k-outerplanar} is complete.
\end{proof}

\section{Kelly subposets in doubly exposed posets}\label{sec:kelly-subposets}

This section is devoted to proving Lemma~\ref{lem:se-vs-kelly}. The setting of this lemma
is the same as in~\cite[Lemma 16]{kozik2019dimension}, and the initial part of our
proof overlaps with the proof from~\cite{kozik2019dimension},
but the core of our argument requires us to investigate
the structure of doubly exposed standard examples in much greater detail, and is completely independent
of the work in~\cite{kozik2019dimension}.

We work with doubly exposed standard examples in a poset
with a planar cover graph. When a standard example is doubly exposed by a pair
\((x_0, y_0)\), we find it convenient to work in a setting
where the vertices \(x_0\) and \(y_0\) are of degree \(1\) in the cover graph.
After slightly modifying the poset, we can ensure that this is the case: 
If a standard example is doubly exposed by a pair \((x, y)\), we can
add elements \(x_0\) and \(y_0\) such that \(x_0\) is a minimal element covered
only by \(x\) and \(y_0\) is a maximal element which covers only \(y\).
The cover graph of the modified poset can be obtained from the cover graph
of the original poset by attaching new degree-\(1\) vertices \(x_0\) and \(y_0\)
to \(x\) and \(y\) respectively. Hence, each standard example which is doubly exposed
by \((x, y)\) in the original poset, is doubly exposed by \((x_0, y_0)\) in
the new poset.

Therefore, throughout this section we assume that \(P\) is a fixed poset with a planar
cover graph \(G\) and we fix a planar drawing of \(G\) and two elements
\(x_0 \in \Min(P)\) and \(y_0 \in \Max(P)\) which lie on the exterior face of
the drawing and which have degree \(1\) in \(G\).
We assume that \(x_0\) is drawn at the bottom
and \(y_0\) on top of the drawing. This does not play a role in the proof, but
justifies the notions of left and right introduced later in a proof.

We note that in the proof of Lemma~\ref{lem:se-vs-kelly} we need to find
a Kelly subposet which does not contain the newly added elements
\(x_0\) and \(y_0\), as they do not belong to the original poset.

We call a path with the endpoints \(u\) and \(v\) a \emph{\(u\)--\(v\) path}. 
For a tree \(R\) and two vertices \(u\) and \(v\) of \(R\),
we denote by \(u R v\) the unique \(u\)--\(v\) path in \(R\).
More generally, if trees
\(R_1\), \ldots, \(R_p\) are subgraphs of \(G\) and
\(u_0\), \ldots \(u_p\) are vertices such that
\(\{u_{i-1}, u_i\} \subseteq V(R_i)\) for \(i \in [p]\),
then we let \(u_0 R_1 \cdots R_p u_p\) denote the union
of the paths \(u_{i-1} R_i u_i\) with \(i \in [p]\).
Whenever we use this notation, it denotes a
path or a cycle.

Let \(A = \Down_P(y_0) \cap \Min(P)\) and
\(B = \Up_P(x_0) \cap \Max(P)\), so that
\(\Inc(A, B)\) is the maximal subset of \(\Inc(\Min(P), \Max(P))\)
which is doubly exposed by \((x_0, y_0)\).
Let us fix a rooted tree \(T\) which is a subgraph of \(G\) such that
\begin{enumerate}
    \item the root of \(T\) is \(x_0\),
    \item the set of leaves of \(T\) is \(B\), and
    \item for every \(b \in B\), the path \(x_0 T b\)
    is a witnessing path from \(x_0\) to \(b\).
\end{enumerate}
Analogously,
let us fix a rooted tree \(S\) which is a subgraph of \(G\) such that
\begin{enumerate}
    \item the root of \(S\) is \(y_0\),
    \item the set of leaves of \(S\) is \(A\),
    and
    \item for every \(a \in A\), the path \(a S y_0\)
    is a witnessing path from \(a\) to \(y_0\).
\end{enumerate}
Following the terminology from%
~\cite{kozik2019dimension}, we refer to \(T\) as the \emph{blue tree}
and to \(S\) as the \emph{red tree}.
The blue and red trees may intersect, but in
one of the first steps of the
proof, we will show that from any doubly exposed standard example you can select
a standard example of linear size such that the parts of \(T\) and \(S\) corresponding to it are disjoint.
But first, let us introduce some more notation and basic properties of doubly exposed
standard examples.

Since the root of \(T\) has only
one child, the drawing of \(G\) determines
partial ``clockwise'' orderings of the vertex sets of \(T\).
Formally, we define a partial
order \(\prec_T\) on \(V(T)\) as follows. Let \(u\) and \(v\) be two
distinct nodes of \(T\). If one of those nodes is an
ancestor of the other, the elements \(u\) and \(v\)
are considered incomparable in \(\prec_T\). Otherwise, let
\(w\) denote the lowest common ancestor of \(u\) and
\(v\) in \(T\). We write \(u \prec_T v\) when
in the drawing of \(G\) the paths \(w T x_0\),
\(w T u\) and \(w T v\) leave the vertex \(w\)
in that clockwise order. Otherwise, when these
three paths leave the vertex \(w\) in the anticlockwise
order, we write \(v \prec_T u\).
Since \(x_0\) has only one child in \(T\), the
path \(x_0 T w\) is nontrivial and thus the order
is well-defined.
Note that since the elements of \(B\) are leaves of \(T\),
they are linearly ordered by \(\prec_T\)
The ``clockwise'' order \(\prec_S\) on \(V(S)\) is defined
analogously.
See Figure~\ref{fig:red blue}.

\begin{figure}
    \centering
    \includegraphics{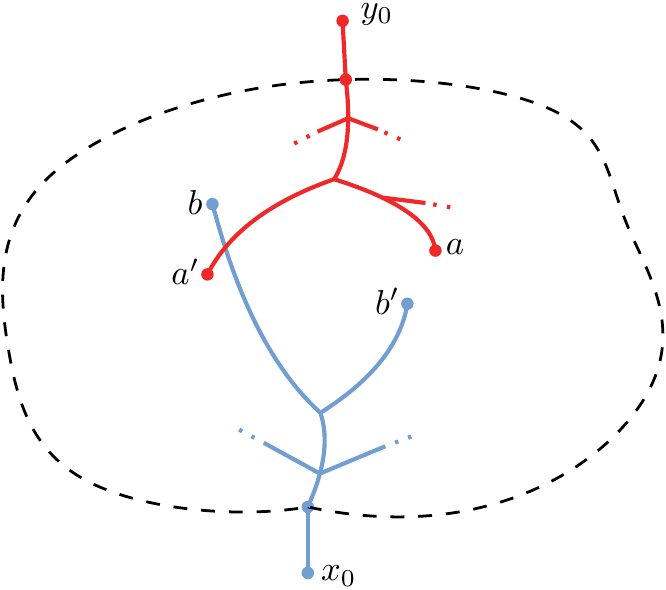}
    \caption{\(b \prec_T b'\) and \(a \prec_S a'\).}
    \label{fig:red blue}
\end{figure}

For every pair of elements
\(a \in A\) and \(b \in B\) with \(a < b\) in \(P\),
fix a witnessing path \(W = W(a, b)\) and two vertices
\(v = v(a, b)\) and \(u = u(a, b)\) on the path \(W\) such that 
\(W = a S v W u T b\)
and the length of the subpath \(v W u\) is smallest possible.
This is well-defined, as we can always take \(v = a\) and \(u = b\).
This way the subpath \(v W u\) is internally disjoint from the
paths \(a S y_0\) and \(x_0 T b\), and if the paths \(a S y_0\) and \(x_0 T b\) intersect,
then \(u = v\).
We let
\[N(a, b) = x_0 T u W v S y_0.\]
See Figure~\ref{fig:N path}.
The path \(x_0 T u\) is called the \emph{blue} portion of \(N(a, b)\),
the path \(u W v\) is called the \emph{black} portion of \(N(a, b)\),
and the path \(v S y_0\) is called the \emph{red} portion of \(N(a, b)\).
Note that the interior of the black portion of the path \(N(a, b)\) may contain vertices
of the trees \(T\) and \(S\).
\begin{figure}
    \centering
    \includegraphics{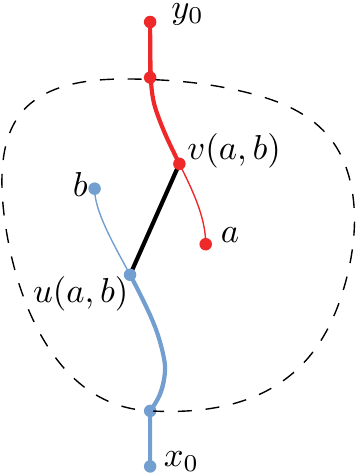}
    \caption{The bolded \(x_0\)--\(y_0\) path is \(N(a, b)\).}
    \label{fig:N path}
\end{figure}

Every cycle \(C\) in \(G\) is represented by a closed curve in the drawing and removing the points on that curve splits the plane into two
parts, one bounded and one
unbounded.
We refer to the bounded part together with the points
on the curve representing \(C\) as the \emph{region
bounded by} \(C\).
When referring to the bounded part without the points
lying on the curve, we explicitly
write about the interior of the region bounded by \(C\).
Clearly, if a connected subgraph of \(G\) contains
a vertex in the region bounded by \(C\)
and a vertex outside the interior of the region bounded
by \(C\), then that subgraph intersects the cycle \(C\).

For any two elements
\(b, b' \in B\), we say that \(b\) is \emph{enclosed by} 
\(b'\) if there exists a cycle \(C\) in \(G\) with \(V(C) \subseteq \Down_P(b')\)
such that \(b\) lies in the region
bounded by \(C\). Analogously, for any two
elements \(a, a' \in A\), we say that \(a\) is
\emph{enclosed by} \(a'\)
if there exists a cycle \(C\) in \(G\) with
\(V(C) \subseteq \Up_P(a')\) such that \(a\) lies in the region bounded by \(C\).

\begin{lemma}\label{lem:enclosed}
  If a standard example \(\{(a_1, b_1), \ldots, (a_n, b_n)\} \subseteq \Inc(A, B)\)
  is doubly exposed by \((x_0, y_0)\),
  then there do not exist distinct \(i, j \in [n]\) such that
  \(b_i\) is enclosed by \(b_j\) or \(a_i\) is enclosed by \(a_j\).
\end{lemma}
\begin{proof}
  We only prove that \(b_i\) is not enclosed by \(b_j\), as
  the proof for \(a_i\) and \(a_j\) is completely analogous.
  Suppose towards a contradiction, that there exist distinct 
  \(i, j \in [m]\) and a cycle \(C\) in \(G\)
  with \(V(C) \subseteq \Down_P(b_j)\) such that \(b_i\) lies
  in the region bounded by \(C\).
  Since \(y_0\) lies on the exterior face of the drawing of \(G\),
  it does not lie in the interior of the region bounded by
  \(C\).
  As \(\Up_P(a_j)\) induces a connected subgraph of \(G\),
  containing both \(b_i\) and \(y_0\), it must intersect
  the cycle \(C\) in a vertex \(w\).
  Since \(V(C) \subseteq \Down_P(b_j)\), we have
  \(a_j \le w \le b_j\) in \(P\), a contradiction.
\end{proof}
\begin{lemma}\label{lem:three-paths}
  Let \(\{(a_1, b_1), \ldots, (a_n, b_n)\} \subseteq \Inc(A, B)\)
  be a standard example,
  let \(i, j, k \in [n]\), and let \(W\) be a
  witnessing path in \(P\).
  \begin{enumerate}
      \item\label{itm:three-paths-b}
      If \(b_i \prec_T b_j \prec_T b_k\) and \(W\) intersects
      both \(x_0 T b_i\) and \(x_0 T b_k\), then \(W\) intersects \(x_0 T b_j\).
      \item\label{itm:three-paths-a}
      If \(a_i \prec_S a_j \prec_S a_k\) and \(W\) intersects
      both \(a_i S y_0\) and \(a_k S y_0\), then \(W\) intersects \(a_j S y_0\).
  \end{enumerate}
\end{lemma}
\begin{proof}
  The items~\eqref{itm:three-paths-b} and~\eqref{itm:three-paths-a} are symmetric,
  so we only prove~\eqref{itm:three-paths-b}.
  Let \(W'\) denote a shortest subpath of \(W\) which intersects both
  \(x_0 T b_i\) and \(x_0 T b_k\).
  Then \(W'\) is a path between the paths \(x_0 T b_i\) and \(x_0 T b_k\)
  which is internally disjoint from these paths.
  Let \(u_i\) and \(u_k\) denote the ends
  of \(W'\) lying on \(x_0 T b_i\) and \(x_0 T b_k\) respectively.
  Suppose towards a contradiction that \(W'\) is disjoint from
  \(x_0 T b_j\). This in particular implies that neither \(u_i\) nor
  \(u_k\) is an ancestor of \(b_j\) in \(T\).
  Hence \(u_i \prec_T b_j \prec_T u_k\) and \(b_j\) lies in the region
  bounded by the cycle \(C = u_i S u_k W' u_i\).
  Let \(w\) denote the lowest common ancestor of \(u_i\) and \(u_k\) in \(T\).
  If \(u_i < u_k\) in \(P\), then the cycle \(C\) is the union
  of the witnessing paths \(w T u_i W' u_k\) and \(w T u_k\),
  so \(V(C) \subseteq \Down_P(u_k) \subseteq \Down_P(b_k)\) and \(b_j\) is enclosed
  by \(b_k\).
  Otherwise, \(C\) is the union of the witnessing paths \(w T u_i\) and
  \(w T u_k W' u_i\), so \(V(C) \subseteq \Down_P(u_i) \subseteq \Down_P(b_i)\) and \(b_j\) is enclosed by \(b_i\).
  In both cases we obtain a contradiction with Lemma~\ref{lem:enclosed}.
\end{proof}

Since \(x_0\) and \(y_0\) lie on the exterior face of
\(G\), every \(x_0\)--\(y_0\) path \(N\) in \(G\)
splits \(G\) into two parts: ``left'' and ``right''.
Formally, let \(u \in V(G) \setminus V(N)\) and choose a path \(M\)
between \(u\) and a vertex \(w \in V(N)\) such that no internal vertex of \(M\) belongs to \(N\).
Note that \(w\) must be an internal vertex of \(N\)
as \(x_0\) and \(y_0\) are of degree one in \(G\).
Since the drawing of \(G\) is planar and the vertices
\(x_0\) and \(y_0\) lie on the exterior face,
either for every choice of \(M\) the paths
\(w N x_0\), \(w M u\) and \(w N y_0\) leave the
vertex \(w\) in that clockwise order, or for every choice
of \(M\) the paths \(w N x_0\), \(w M u\) and \(w N y_0\)
leave the vertex \(w\) in that anticlockwise order.
In the former case we say that \(u\) is \emph{left of}
\(N\) and in the latter case, we say that \(u\) is
\emph{right of} \(N\). For instance, in Figure~\ref{fig:N path},
\(a\) is right of \(N(a, b)\) and \(b\) is left of \(N(a, b)\).

\begin{lemma}\label{lem:ijk}
  Let \(\{(a_1, b_1), \ldots, (a_n, b_n)\} \subseteq \Inc(A, B)\)
  be a standard example and
  let \(i, j, k \in [n]\).
  If
  \(a_i \prec_S a_j\) and \(b_j \prec_T b_k\),
  then the intersection of the paths
  \(N(a_i, b_j)\) and \(N(a_j, b_k)\)
  contains a vertex which belongs to the red or the
  black portion of \(N(a_i, b_j)\) and
  to the black or the blue portion of
  \(N(a_j, b_k)\).
\end{lemma}
\begin{proof}
  In the tree \(S\), 
  the vertex \(v(a_i, b_j)\) is an ancestor of
  \(a_i\) but not of \(a_j\) because that would imply
  \(a_j \le v(a_i, b_j) \le b_j\) in \(P\).
  As \(a_i \prec_S a_j\), this implies that
  \(v(a_i, b_j) \prec_S a_j\).
  Furthermore, the blue and the black parts of
  \(N(a_i, b_j)\) have all their vertices
  in \(\Down_P(b_j)\).
  Since \(a_j\) is incomparable with \(b_j\) in \(P\),
  this implies that no vertex of \(a_j S y_0\)
  is right of \(N(a_i, b_j)\).
  In particular,
  \(v(a_j, b_k)\) is not right of \(N(a_i, b_j)\).
  
  If \(u(a_i, b_j)\) is an ancestor of
  \(u(a_j, b_k)\) in \(T\), then
  \(u(a_i, b_j)\) lies on the black portion
  of \(N(a_i, b_j)\) and the blue portion of
  \(N(a_j, b_k)\), so the claim is satisfied.
  Hence we assume that \(u(a_i, b_j)\) is not
  an ancestor of \(u(a_j, b_k)\) in \(T\).
  Furthermore, \(u(a_j, b_k)\) is not an ancestor
  of \(u(a_i, b_j)\) as that would imply
  \(a_j \le u(a_j, b_k) \le u(a_i, b_j) \le b_j\)
  in \(P\).
  Since \(b_j \prec_T b_k\), this implies that
  \(u(a_i, b_j) \prec_T u(a_j, b_k)\).
  Let \(u'\) denote the least vertex of
  \(x_0 T u(a_j, b_k)\)
  which does not lie on \(x_0 T u(a_i, b_j)\).
  Since \(u(a_i, b_j) \prec_T u(a_j, b_k)\),
  \(u'\) is not left of
  \(N(a_i, b_j)\). (Possibly \(u'\) lies on the black
  or the red part of \(N(a_i, b_j)\).)
  Hence, the \(u'\)--\(v(a_j, b_k)\) subpath
  of \(N(a_j, b_k)\) intersects \(N(a_i, b_j)\)
  in a vertex \(w\). See Figure~\ref{fig:ijk-claim}.
  \begin{figure}
      \centering
      \includegraphics{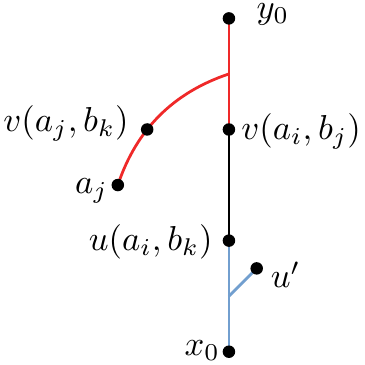}
      \caption{Illustration of Lemma~\ref{lem:ijk}.
      The vertical \(x_0\)--\(y_0\) path is \(N(a_i, b_j)\). The
      \(u'\)--\(v(a_j, b_k)\) subpath of \(N(a_j, b_k)\)
      has to intersect \(N(a_i, b_j)\).}
      \label{fig:ijk-claim}
  \end{figure}
  
  If \(w\) lies on the
  \(u'\)--\(u(a_j, b_k)\) subpath of
  \(N(a_j, b_k)\), then
  \(w\) lies on the blue part of \(N(a_j, b_k)\)
  and by definition of \(u'\) it does not
  lie on the blue part of \(N(a_i, b_j)\), so
  \(w\) satisfies the claim.
  Otherwise, \(w\) has to lie on the
  \(u(a_j, b_k)\)--\(v(a_j, b_k)\)
  subpath of \(N(a_j, b_k)\).
  Then \(w\) lies on the black portion of
  \(N(a_j, b_k)\) and on the red portion of
  \(N(a_i, b_j)\), lest \(a_j \le w \le b_j\)
  holds in \(P\).
  This completes the proof of the claim.
\end{proof}

\begin{lemma}\label{lem:iff}
  Let \(\{(a_1, b_1), \ldots, (a_n, b_n)\} \subseteq \Inc(A, B)\)
  be a standard example and let \(i, j \in [n]\).
  Then we have \(a_i \prec_S a_j\)
  if and only if \(b_i \prec_T b_j\).
\end{lemma}
\begin{proof}
  Suppose to the contrary that the claim does not hold.
  Then there exists indices \(i, j \in [m]\)
  with \(a_i \prec_S a_j\) and \(b_j \prec_T b_i\).
  Apply Lemma~\ref{lem:ijk} with \(k = i\).
  All vertices on the red and the black portions of
  \(N(a_i, b_j)\) belong to \(\Up_P(a_i)\) and all
  vertices on the black and the blue portions of \(N(a_i, b_j)\)
  belong to \(\Down_P(b_i)\).
  Hence \(a_i \le w \le b_i\) holds in \(P\),
  a contradiction.
\end{proof}

By Lemma~\ref{lem:iff},
the pairs of any standard example in \(\Inc(A, B)\) can be listed in an
order \((a_1, b_1)\), \ldots, \((a_n, b_n)\) such that
\(a_1 \prec_S \cdots \prec_S a_n\) and \(
b_1 \prec_T \cdots \prec_T b_n\).

For a standard example \(I \subseteq \Inc(A, B)\),
we define an auxiliary digraph \(D(I)\) with the vertex set \(I\), where
a pair \((a, b) \rightarrow (a', b')\) is an arc of \(D(I)\) when the paths \(a S y_0\)
and \(x_0 T b'\) intersect.
A path \((a_1, b_1) \rightarrow \cdots \rightarrow (a_p, b_p)\) in \(D(I)\) is called
\emph{increasing} if \(b_1 \prec_T \cdots \prec_T b_p\) (and thus \(a_1 \prec_T \cdots \prec_T a_p\)) and \emph{decreasing} if
\(b_p \prec_T \cdots \prec_T b_1\) (and thus \(a_p \prec_T \cdots \prec_T a_1\)).
In Figure~\ref{fig:incr-6-path} you can see an increasing
path on \(6\) pairs.
\begin{figure}
    \centering
    \includegraphics{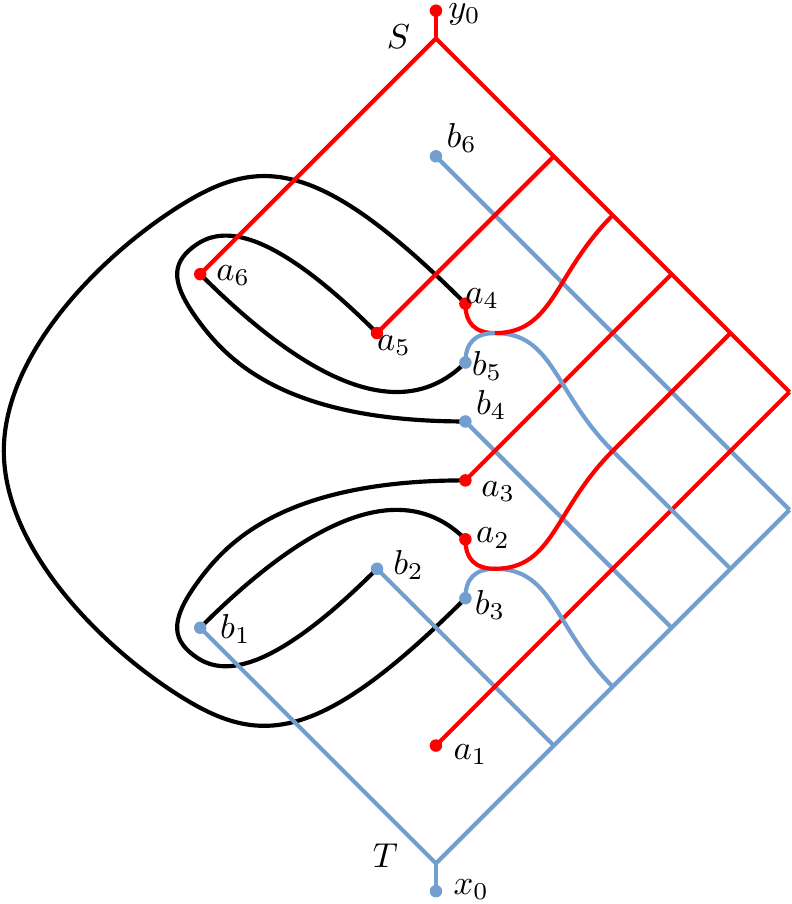}
    \caption{An increasing path \((a_1, b_1) \rightarrow \cdots
    \rightarrow (a_6, b_6)\).
    The paths \(W(a_{i+1}, b_i)\) are drawn to the left and curved.
    The union of these paths and the trees \( S \) and \( T \) contains
    a witnessing path from \(a_i\) to \(b_j\) for each pair of distinct \(i\)
    and \(j\).
    }
    \label{fig:incr-6-path}
\end{figure}

\begin{lemma}\label{lem:incr-path}
  For every standard example \(I \subseteq \Inc(A, B)\),
  every increasing or decreasing path in the digraph
  \(D(I)\) consists of at most \(6\) pairs.
\end{lemma}
\begin{proof}
  Because of symmetry, it suffices to show that every increasing path has
  at most \(6\) vertices.
  Suppose to the contrary that \(D(I)\) contains an increasing path
  \((a_1, b_1) \rightarrow \cdots  \rightarrow (a_7, b_7)\).
  We have \(a_1 \prec_S \cdots \prec_S a_7\), \(b_1 \prec_T \cdots \prec_T b_7\),
  and \(a_i S y_0 \cap x_0 T b_{i+1} \neq \emptyset\) for \(1 \le i \le 6\).
  For every \(i\) with \(1 \le i \le 6\), let \(c_i = u(a_i, b_{i+1}) = v(a_i, b_{i+1})\)
  denote the only vertex of the black portion of the path
  \(N(a_i, b_{i+1})\).
  Clearly, we have \(a_i \le c_i \le b_{i+1}\) in \(P\).
  As the black portions of the paths
  \(N(a_i, b_{i+1})\) and \(N(a_{i+1}, b_{i+2})\) 
  are trivial for \(1 \le i \le 5\),
  Lemma~\ref{lem:ijk} applied to \(j = i+1\) and \(k = i+2\) implies that
  the red portion of
  \(N(a_i, b_{i+1})\) intersects the blue portion of
  \(N(a_{i+1}, b_{i+2})\),
  that is \(c_i S y_0 \cap x_0 T c_{i+1} \neq \emptyset\).
  In particular, we have
  \[
    x_0 \le c_1 \le \cdots \le c_6 \le y_0\textrm{ in \(P\).}
  \]

  Let \(W_0\) be a witnessing path from \(x_0\) to \(y_0\) containing
  the elements \(c_1\), \ldots, \(c_6\) which satisfies
  \(x_0 W_0 c_1 = x_0 T c_1\) and \(c_6 W_0 y_0 = c_6 S y_0\).
  For each \(i\) with \(1 \le i \le 7\), let \(t_i\) denote the greatest element of
  \(W_0\) which lies on \(x_0 T b_i\) and let \(s_i\) denote the least
  element of \(W_0\) which lies on \(a_i S y_0\).
  As \(c_i\) lies on both \(x_0 T b_{i+1}\) and \(a_i S y_0\),
  we have \(s_i \le c_i \le t_{i+1}\) in \(P\) for
  \(1 \le i \le 6\).
  Furthermore, we have \(t_i < s_i\) in \(P\) for \(1 \le i \le 7\)
  as otherwise we would have \(a_i \le s_i \le t_i \le b_i\) in \(P\).
  Hence
  \[
    x_0 \le t_1 < s_1 \le c_1 \le t_2 < s_2 \le c_2 \le \cdots \le c_6 \le t_7 < s_7 \le y_0
  \]
  holds in \(P\).

  Let us now prove a sequence of subclaims.

  \begin{nclaim}\label{clm:disjoint-path}
    For \(1 \le i < j \le 7\), the witnessing path \(W(a_j, b_i)\) is
    disjoint from \(W_0\).
  \end{nclaim}
  \begin{subproof}
  Suppose to the contrary that \(W(a_j, b_i)\) intersects \(W_0\)
  in a vertex \(w\).
  Then we have \(a_j \le w \le b_i\) in \(P\) and \(w\) is
  comparable with \(s_i\) in \(P\).
  If \(w \le s_i\) in \(P\), then we have \(a_j \le w \le s_i \le s_j \le b_j\) in \(P\),
  which is impossible.
  Similarly, if \(w > s_i\) in \(P\), then we have
  \(a_i \le s_i < w \le b_i\) in \(P\), which is again impossible.
  As both cases lead to a contradiction, the proof follows.
  \end{subproof}
  
  \begin{nclaim}\label{clm:a_j b_i left of W_0}
    The vertices \(b_1\), \ldots, \(b_6\), \(a_2\), \ldots, \(a_7\)
    are left of \(W_0\).
  \end{nclaim}
\begin{subproof}

  Since \(t_1 < s_1 \le c_1\) in \(P\), the path
  \(x_0 W_0 t_1\) is a proper subpath of \(x_0 W_0 c_1\),
  which in turn is a subpath of \(x_0Tb_2\) by our choice of
  \(W_0\).
  As \(b_1 \prec_T b_2\), this implies that \(b_1\) is left of
  \(W_0\). 
  By Claim~\ref{clm:disjoint-path}, each path \(W(a_j, b_1)\) with 
  \(2 \le j \le 7\), is disjoint from \(W_0\).
  As the end \(b_1\) of \(W(a_j, b_1)\) is left of \(W_0\),
  the end \(a_j\) must be left of \(W_0\) as well.
  Hence the vertices \(a_2\), \ldots, \(a_7\) are left of \(W_0\).
  Similarly, for every \(i\) with \(1 \le i \le 6\),
  the path \(W(a_7, b_i)\) contains the vertex \(a_7\) left of
  \(W_0\) and is  
  disjoint from \(W_0\), so the vertices \(b_1\), \ldots, \(b_6\)
  must be left of \(W_0\). This proves the claim.
\end{subproof}

  From this point on we focus on the elements \(a_i\) and \(b_i\)
  with \(2 \le i \le 6\).
  
  \begin{nclaim}\label{clm:disjoint paths in S and T}
    The paths \(t_2 T b_2\), \ldots, \(t_6 T b_6\) are pairwise
    disjoint and the paths \(a_2 S s_2\), \ldots, \(a_6 S s_6\)
    are pairwise disjoint.
  \end{nclaim}
  \begin{subproof}
  Suppose to the contrary
  that there exist \(i\) and \(j\) with \(2 \le i < j \le 6\)
  such that the paths \(t_i T b_i\)
  and \(t_j T b_j\) intersect in a vertex \(w\).
  Then we have \(a_i \le s_i \le t_j \le w \le b_i\) in \(P\), which is
  a contradiction.
  A similar argument shows that the paths \(a_2 S s_2\), \ldots, \(a_6 S s_6\)
  are pairwise disjoint.
  The claim follows.  
  \end{subproof}
  
  \begin{nclaim}\label{clm:dp2}
    For \(2 \le i \le 6\) and \(2 \le j \le 6\),
    the paths \(t_i T b_i\) and \(a_j S s_j\) 
    are disjoint except possibly having a common end on \(W_0\).
  \end{nclaim}
  
\begin{subproof}
  Let \(i\) and \(j\) be such that \(2 \le i \le 6\) and \(2 \le j \le 6\).
  Suppose towards a contradiction that \(t_i T b_i\) and \(a_j S s_j\) intersect
  in a vertex \(w\) which does not lie on \(W_0\).
  It is impossible that \(i > j\) as that would imply \(t_i < w < s_j \le t_i\) in \(P\).
  Furthermore, it is impossible that \(i = j\) as then we would have
  \(a_i \le w \le b_i\) in \(P\).
  Let us hence assume that \(i < j\). Let \(N = N(a_j, b_i)\).
  Since the paths \(t_i T b_i\) and \(a_j S s_j\) intersect,
  \(N\) is a witnessing path from \(x_0\) to \(y_0\) and the black part
  of \(N\) is trivial.
  The ends of the path \(t_i N s_j\) lie on \(W_0\)
  and all internal vertices of \(t_i N s_j\) are left of \(W_0\).
  Furthermore, by Claim~\ref{clm:a_j b_i left of W_0},
  all vertices of the path \(t_j T b_j\) except
  \(t_j\) are left of \(W_0\).
  As \(t_j\) is an internal vertex of the path \(t_i W_0 s_j\),
  this implies that unless \(t_j T b_j\) intersects
  the path \(t_i N s_j\), \(b_j\) lies in the region bounded by
  the cycle \(C = t_i N s_j \cup t_i W_0 s_j\).
  But by Claims~\ref{clm:disjoint paths in S and T} and%
  ~\ref{clm:dp2} the path \(t_j T b_j\) does not intersect \(t_i N s_j\).
  Hence \(b_j\) indeed must lie in the region bounded by \(C\).
  As \(j \le 6\), we have \(s_j \le t_7 \le b_7\) in \(P\),
  so \(V(C) \subseteq \Down_P(s_j) \subseteq \Down_P(b_7)\),
  that is \(b_j\) is enclosed by \(b_7\).
  This contradicts Lemma~\ref{lem:enclosed} and
  completes the proof.
\end{subproof}

  The Claims~\ref{clm:a_j b_i left of W_0}--\ref{clm:dp2}
  show that the union of \(W_0\) and the paths \(t_i T b_i\) and
  \(a_i S s_i\) with \(2 \le i \le 6\) is a tree with no vertices right
  of \(W_0\). The vertices \(x_0\), \(b_2\), \(a_2\), \ldots,
  \(b_7\), \(a_7\), \(y_0\) are leaves of that tree.
  
  \iffalse
  , but they do not have
  to appear in the drawing in that clockwise order:
  When for some \(i\) we have \(s_i = t_{i+1}\), it is possible that
  \(a_i\) and \(b_{i+1}\) are swapped in this order. 
  See Figure~\ref{fig:6 path}.
  \begin{figure}
    \centering
    \begin{subfigure}[b]{0.44\textwidth}
        \centering
        \includegraphics{figs/6-path.pdf}
        \caption{If for some \(i\) we have \(s_i = t_{i+1}\), then it is
        possible that the path \(t_{i+1} T b_{i+1}\) goes ``above'' the path \(a_i S s_i\).}
        \label{fig:6 path}
    \end{subfigure}
    \hfill
    \begin{subfigure}[b]{0.54\textwidth}
        \centering
      \includegraphics{figs/246a.pdf}
      \qquad
      \includegraphics{figs/246b.pdf}
      \caption{Two possible outcomes of Claim~\ref{clm:246}. \vspace{2.25cm}}
      \label{fig:246}
    \end{subfigure}
    \caption{Figures illustrating the setting of Claim~\ref{clm:246}.}
\end{figure}
    \fi
    \begin{nclaim}\label{clm:246}
    Either
    \(W(a_4, b_2)\) intersects \(a_6 S s_6\), or
    \(W(a_6, b_4)\) intersects \(t_2 T b_2\). See Figure~\ref{fig:246}.
  \end{nclaim}
  \begin{figure}[b]
        \centering
      \includegraphics{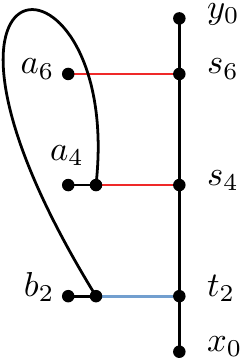}
      \qquad
      \includegraphics{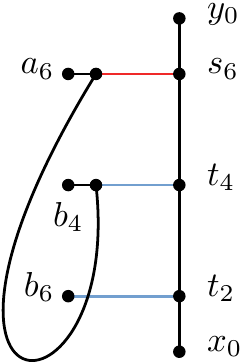}
      \caption{The two possible outcomes of Claim~\ref{clm:246}}
      \label{fig:246}
  \end{figure}
  \begin{subproof}
  Let \(N_1 = N(a_4, b_2)\) and \(N_2 = N(a_6, b_4)\).
   Each of the paths \(t_2 N_1 s_4\) and \(t_4 N_2 s_6\)
   has both its ends on \(W_0\) and all internal vertices left of \(W_0\).
   Since \(t_2 < t_4 < s_4 < s_6\), the paths \(t_2 N_1 s_4\) and \(t_4 N_2 s_6\)
   have to intersect in a vertex \(w\) which does not lie on \(W_0\).
   By Claims~\ref{clm:disjoint paths in S and T} and~\ref{clm:dp2},
   the vertex \(w\) has to lie on the black portion of \(N_1\) or \(N_2\). 
   Suppose that \(w\) lies on the black portion of \(N_1\) (and thus lies on
   \(W(a_4, b_2)\)). It is impossible that \(w\) lies on the blue or the
   black portion of \(N_2\) as that would imply
   \(a_4 \le w \le u(a_6, b_4) \le b_4\) in \(P\).
   Hence \(w\) must lie on \(a_6 S s_6\).
   A symmetric argument shows that if \(w\) lies on the black portion of \(N_2\),
   then \(w\) lies on the intersection of \(W(a_6, b_4)\) and \(t_2 T b_2\).
   The proof follows.
  \end{subproof}

  The two alternatives in the statement of Claim~\ref{clm:246}
  are symmetric, so without loss of generality we
  assume that \(W(a_4, b_2)\) intersects \(a_6 S s_6\).
  Let \(W = W(a_4, b_2)\), \(u = u(a_4, b_2)\) and
  \(v = v(a_4, b_2)\).
  Furthermore, let \(v'\) denote the greatest element of the
  intersection of \(W\) with \(a_6 S s_6\).
  This way the paths \(v S s_4\) and
  \(v' S s_6\) have their ends on the paths \(W\) and
  \(W_0\) but are otherwise disjoint \(W\) and \(W_0\).
  Furthermore, by Claim~\ref{clm:disjoint paths in S and T}, the paths
  \(v S s_4\) and \(v' S s_6\)
  are disjoint, so the union of
  the witnessing paths \(v W v' S s_6\) and
  \(v S s_j W_0 s_6\) forms a cycle, which we denote by
  \(C\). 
  See Figure~\ref{fig:fc}.

  \begin{figure}[hb]
    \centering
    \includegraphics{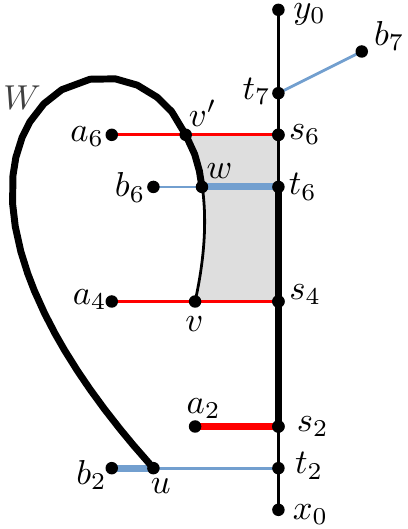}
    \caption{
    The shaded area is
    the region bounded by the cycle \(C\).} 
    \label{fig:fc}
  \end{figure}

  We have \(s_6 \le t_7 \le b_7\) in \(P\),
  so \(V(C) \subseteq \Down_P(s_6) \subseteq \Down_P(b_7)\).
  By Lemma~\ref{lem:enclosed}, \(b_6\) is not enclosed by
  \(b_7\), so \(b_6\) does not lie in the region bounded
  by \(C\) in the drawing.
  As \(s_4 < t_6 < s_6\) in \(P\)
  and \(t_6\) is the only vertex of \(t_6 T b_6\)
  which lies on \(W_0\),
  the path \(t_6 T b_6\) has to intersects the cycle \(C\) in a vertex
  \(w\) which does not lie on \(W_0\).
  By Claim~\ref{clm:dp2}, \(w\) does not lie on \(a_4 S s_4\)
  nor \(a_6 S s_6\).
  Hence \(w\) has to lie on \(v W v'\).
  This implies that
  \(a_2 \le b_2\) holds in \(P\)
  as witnessed by the path
  \(a_2 S s_2 W_0 t_6 T w W b_2\),
  a contradiction.
  Hence there does not exist an increasing path consisting of
  \(7\) pairs.
  This completes the proof of the lemma.
\end{proof}

For a standard example \(I\) in \(\Inc(A, B)\), we denote
by \(T(I)\) and \(S(I)\) the subtrees of \(T\) and
\(S\) respectively, defined as follows:
\[
T(I) = \bigcup_{(a, b) \in I} x_0 T b\quad\text{and}\quad
S(I) = \bigcup_{((a, b) \in I} a S y_0.
\]
A standard example \(I \subseteq \Inc(A, B)\) is
said to be \emph{separated} if the tress \(T(I)\) and
\(S(I)\) are disjoint.

\begin{lemma}\label{lem:disjoint-subtrees}
  For every positive integer \(n\), if \(I \subseteq \Inc(A, B)\)
  is a standard example with \(|I| \ge 36n+1\), then there exists
  a standard example \(I' \subseteq I\) with \(|I| \ge n + 1\)
  which is separated.
\end{lemma}
\begin{proof}
  Let \((a_1, b_1)\), \ldots, \((a_{36n+1}, b_{36n+1})\)
  be any \(36n+1\) elements of \(I\) listed in the order such that
  \[
  b_1 \prec_T \cdots \prec_T b_{36n+1}
  \quad\textrm{ and }\quad
  a_1 \prec_S \cdots \prec_S a_{36n+1}.
  \]
  For each \(i \in [36n+1]\),
  let \(p(i)\) denote the length of
  a longest increasing path in \(D(I)\) which starts
  with \((a_i, b_i)\) and let \(q(i)\) denote the length of
  a longest decreasing path in \(D(I)\) which starts
  with \((a_i, b_i)\).
  In particular, for every arc \((a_i, b_i) \rightarrow (a_j, b_j)\) in \(D(I)\),
  if \(i < j\) then \(p(i) > p(j)\), and
  if \(i > j\) then \(q(i) > q(j)\).
  As \(D(I)\) does not have loops, this implies that
  for \(i, j \in [36n+1]\), if \(p(i) = p(j)\)
  and \(q(i) = q(j)\), then \((a_i, b_i) \rightarrow (a_j, b_j)\) is not
  an arc of \(D(I)\).
  
  By Lemma~\ref{lem:incr-path}, for each \(i \in [36n+1]\) we have
  \(p(i), q(i) \in \{1, \ldots, 6\}\).
  Hence, by the pigeonhole principle, there exists a subset
  \(X \subseteq [36n+1]\) with \(|X| \ge n+1 \) such that 
  the pair of the values \((\alpha(j), \beta(j))\)
  is the same for all \(j \in X\).
  For such \(X\), there do not exist \(i, j \in X\).
  such that \((a_i, b_i) \rightarrow (a_j, b_j)\) is an arc of \(D(I)\).
  Hence the standard example \(\{(a_i, b_i) : i \in X\}\)
  satisfies the lemma.
\end{proof}

\begin{lemma}\label{lem:domino}
  Let \(I = \{(a_1, b_1), \ldots, (a_n, b_n)\} \subseteq \Inc(A, B)\) be a separated standard example
  with
  \[
  b_1 \prec_T \cdots \prec_T b_n
  \quad\textrm{ and }\quad
  a_1 \prec_S \cdots \prec_S a_n.
  \]
  Let \(i, j \in \{2, \ldots, n - 1\}\) be distinct.
  \begin{enumerate}
      \item If \(b_{j-1}\) is left of \(N(a_i, b_j)\), then all
  \(b_{1}\), \ldots, \(b_{j-1}\) are left of \(N(a_i, b_j)\).
      \item\label{itm:domino-1} If \(b_{j+1}\) is right of \(N(a_i, b_j)\), then all
  \(b_{j+1}\), \ldots, \(b_{n}\) are right of \(N(a_i, b_j)\).
      \item If \(a_{i+1}\) is left of \(N(a_i, b_j)\), then all
  \(a_{i+1}\), \ldots, \(a_{n}\) are left of \(N(a_i, b_j)\).
      \item If \(a_{i-1}\) is right of \(N(a_i, b_j)\), then all
  \(a_{1}\), \ldots, \(a_{i-1}\) are right of \(N(a_i, b_j)\).
  \end{enumerate}
\end{lemma}
\begin{proof}
  The four statements are symmetric and thus we only prove~\eqref{itm:domino-1}.
  Let \(N = N(a_i, b_j)\), \(u = u(a_i, b_j)\), and \(v = v(a_i, b_j)\), and assume
  that \(b_{j+1}\) is right of \(N\).
  We need to show that for every \(k \in \{j+1, \ldots, n\}\),
  the vertex \(b_k\) is right of \(N\).
  We already assumed that it holds for \(k = j + 1\), so
  let us assume that \(k \ge j + 2\)
  and suppose towards a contradiction that
  \(b_k\) is not right of \(N\).
  Then \(b_k\) has to be left of \(N\).
  Since the standard example is separated,
  \(x_0 T b_k\) is disjoint from \(v S y_0\).
  As \(b_j \prec_T b_k\), this implies that the path
  \(x_0 T b_k\) has to intersect \(u N v\).
  Let \(w_k\) denote the least vertex of the intersection of \(x_0 T b_k\)
  with \(u N v\) and let \(N' = x_0 T w_k N y_0\).
  \begin{nclaim}\label{clm:wkTbk}
    All vertices of \(w_k T b_k\) except \(w_k\) are left of \(N'\).
  \end{nclaim}
  \begin{subproof}
    The path \(N'\) is the union of paths \(x_0 T w_k\), \(w_k N v\) and \(v S y_0\).
    The path \(x_0 T w_k\) intersects \(w_k T b_k\) only in \(w_k\) because \(w_k\) lies on
    \(x_0 T b_k\).
    The path \(w_k N v\) intersects \(w_k T b_k\) only in \(w_k\) because
    \(V(w_k N v) \subseteq \Down_P(w_k)\) and \(V(w_k T b_k) \subseteq \Up_P(w_k)\).
    Finally, the path \(v S y_0\) is disjoint from \(w_k T b_k\) because the standard example is separated.
    Hence \(w_k\) is the only vertex of \(w_k T b_k\) lying on \(N'\).
    As \(b_k\) is left of \(N\), it is also left of \(N'\), so the claim must hold.
  \end{subproof}
  The witnessing path \(u N w_k\) intersects both \(x_0 T b_j\) and \(x_0 T b_k\).
  By Lemma~\ref{lem:three-paths}, \(u N w_k\) intersects \(x_0 T b_{j+1}\), say
  in a vertex \(w_{j+1}\), and we have
  \[
    w_k \le w_{j+1} \le u \textrm{ in }P.
\]
  If \(b_{j+1}\) were not right of \(N'\), then the fact that \(b_{j+1}\) is right of \(N\)
  would imply that \(b_{j+1}\) lies in the region bounded by the cycle \(C = u T w_k N u\).
  As \(V(C) \subseteq \Down_P(u) \subseteq \Down_P(b_j)\), the vertex \(b_{j+1}\) would be
  enclosed by \(b_j\), which would contradict Lemma~\ref{lem:enclosed}.
  Let us hence assume that \(b_{j+1}\) is right of \(N'\) and let
  \(z\) denote the least vertex on the path \(w_{j+1} T b_{j+1}\) which is right of
  \(N'\).
  As \(w_{j+1}\) is not right of \(N'\), the parent \(z'\) of \(z\) in \(T\)
  satisfies
  \[
    w_{j+1} \le z' < z\textrm{ in }P
  \]
  and \(z'\) lies on \(N'\).
  It is impossible that \(z'\) lies on \(x_0 T w_k\) because that would imply
  \(b_k \prec_T b_{j+1}\) (by Claim~\ref{clm:wkTbk} this is impossible even if \(z' = w_k\)).
  Since the standard example is separated, it is also impossible that \(z'\) lies on
  \(v S y_0\). Hence \(z'\) has to be an internal vertex of the path \(w_k N v\).
  Hence, we have
  \[
    v < z' < w_k\textrm{ in }P.
  \]
  See Figure~\ref{fig:domino}.
  \begin{figure}
      \centering
      \includegraphics{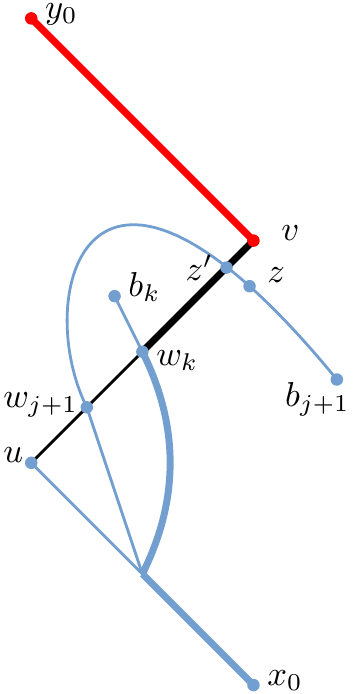}
      \caption{Illustration of Lemma~\ref{lem:domino}. The bolded path is \(N'\).}
      \label{fig:domino}
  \end{figure}
  Summarizing, we have \(w_{j+1} \le z' < w_k \le w_{j+1}\) in \(P\).
  This contradiction completes the proof of the lemma.
\end{proof}

\begin{lemma}\label{lem:domino+1}
  Let \(\{(a_1, b_1), \ldots, (a_{n}, b_{n})\} \subseteq \Inc(A, B)\) be a separated standard example
  with
  \[
  b_1 \prec_T \cdots \prec_T b_n
  \quad\textrm{ and }\quad
  a_1 \prec_S \cdots \prec_S a_n,
  \]
  and let \(i\) and \(j\) satisfy \(1 \le i < j \le n\).
  Then
  \(b_i\) and \(a_j\) are left of \(N(a_i, b_j)\).
\end{lemma}
\begin{proof}
  Since \(a_i\) is incomparable with \(b_i\) in \(P\),
  \(x_0 T b_i\) is disjoint from the black and red portions
  of \(N(a_i, b_j)\).
  In particular, \(u(a_i, b_j)\) is not an ancestor
  of \(b_i\).
  As \(b_i \prec_T b_j\), this implies that
  \(b_i \prec_T u(a_i, b_j)\) and thus
  \(b_i\) is left of \(N(a_i, b_j)\).
  A symmetric argument shows that \(a_j\) is left of
   \(N(a_i, b_j)\) as well.
\end{proof}

\begin{lemma}\label{lem:separating-N-path}
  Let \(\{(a_1, b_1), \ldots, (a_{n}, b_{n})\} \subseteq \Inc(A, B)\) be a separated standard example
  with
  \[
  b_1 \prec_T \cdots \prec_T b_n
  \quad\textrm{ and }\quad
  a_1 \prec_S \cdots \prec_S a_n,
  \]
  and let \(i\), \(j\) and \(k\) satisfy \(1 \le i < j < k \le n\).
  Then
  \begin{enumerate}
      \item\label{itm:separating-N-path-1} \(a_i\) is right of \(N(a_j, b_k)\), or
      \item\label{itm:separating-N-path-2} \(b_k\) is right of \(N(a_i, b_j)\).
  \end{enumerate}
\end{lemma}
\begin{proof}
Suppose to the contrary that the lemma does not hold, that is
neither~\eqref{itm:separating-N-path-1}, nor~\eqref{itm:separating-N-path-2}
holds.
Then, as \(a_i\) does not lie on \(N(a_j, b_k)\),
\(a_i\) must be left of \(N(a_j, b_k)\)
The witnessing path \(a_i S v(a_i, b_j)\) is disjoint from \(N(a_j, b_k)\)
because the standard example is separated and \(a_j\) is incomparable with \(b_j\)
in \(P\).
Hence \(v(a_i, b_j)\) is left of \(N(a_j, b_k)\) as well.
Analogously, the vertices \(b_k\) and \(u(a_j, b_k)\) must be left
of \(N(a_i, b_j)\).

We have \(b_j \prec_T b_k\) and \(u(a_j, b_k)\) is left of \(N(a_i, b_j)\).
As the standard example is separated, the above implies that either
(\Rom{1}) \(u(a_i, b_j)\) is an ancestor of \(u(a_j, b_k)\), or
(\Rom{2}) \(u(a_i, b_j)\) is not an ancestor of \(u(a_j, b_k)\) and the path
\(x_0 T u(a_j, b_k)\) intersects \(N(a_i, b_j)\) in an internal vertex of the black portion.
Either way, the path \(x_0 T u(a_j, b_k)\) intersects the black portion of \(N(a_i, b_j)\)
(if (\Rom{1}) holds then \(u(a_i, b_j)\) is a vertex in the intersection).
Let \(z\) denote the least vertex of \(x_0 T u(a_j, b_k)\) which lies on the black portion of
\(N(a_i, b_j)\). Let \(N' = x_0 T z N(a_i, b_j) y_0\).
The only vertex of \(z T u(a_j, b_k)\)
which lies on \(N'\) is \(z\).
Furthermore, no vertex of \(N'\) is left of \(N(a_i, b_j)\), so the fact that
\(u(a_j, b_k)\) is left of \(N(a_i, b_j)\) implies that \(u(a_j, b_k)\) is also
left of \(N'\). Hence, the paths \(z N' x_0\), \(z T u(a_j, b_k)\), and \(z N(a_i, b_j) v(a_i, b_j)\)
leave the vertex \(z\) in that clockwise order.
Since  \(v(a_i, b_j)\) is not right of \(N(a_j, b_k)\), this implies that the
path \(z N(a_i, b_j) v(a_i, b_j)\) intersects \(N(a_j, b_k)\) in a vertex \(w\) distinct from
\(z\). It is impossible that \(w\) lies on the blue portion of \(N(a_j, b_k)\).
Hence \(w\) lies on the black or red portion of \(N(a_j, b_k)\).
But this implies that
\(a_j \le v(a_j, b_k) \le w \le z \le u(a_i, b_j) \le b_j\) in \(P\), a contradiction.
The lemma follows.
\end{proof}

We proceed to the proof of Lemma~\ref{lem:se-vs-kelly}.

\begin{proof}[Proof of Lemma~\ref{lem:se-vs-kelly}]
Let \(n \ge 3\) and suppose that
\(\Inc(A, B)\) contains a standard example of size \(360n+1\).
We need to show that
\(P - \{x_0, y_0\}\) contains a subposet isomorphic to the Kelly poset \(\Kelly{n}\).
By Lemma~\ref{lem:disjoint-subtrees}, \(\Inc(A, B)\) contains a standard example
of size \(10n+1\) which is separated. Let us fix such a standard example
\(\{(a_1, b_1), \ldots, (a_{10n+1}, b_{10n+1})\}\).
We assume \(a_1 \prec_S \cdots \prec_S a_{10n+1}\) and
\(b_1 \prec_T \cdots \prec_T b_{10n+1}\).
By Lemma~\ref{lem:separating-N-path}, \(a_{5n}\) is right of \(N(a_{5n+1}, b_{5n+2})\)
or \(b_{5n+2}\) is right of \(N(a_{5n}, b_{5n+1})\)

If \(a_{5n}\) is right of \(N(a_{5n+1}, b_{5n+2})\), then by Lemma~\ref{lem:domino}, \(b_{5n+1}\) is left of \(N(a_{5n+1}, b_{5n+2})\).
Hence, by Lemma~\ref{lem:domino+1}, the vertices \(a_1\), \ldots, \(a_{5n}\) are right of
\(N(a_{5n+1}, b_{5n+2})\) and the vertices \(b_1\), \ldots, \(b_{5n}\) are left of \(N(a_{5n+1}, b_{5n+2})\). See Figure~\ref{fig:separating-N-path-b}.

On the other hand, if \(b_{5n+2}\) is right of \(N(a_{5n}, b_{5n+1})\), then by Lemma~\ref{lem:domino},
\(a_{5n+1}\) is left of \(N(a_{5n}, b_{5n+1})\).
Hence, by Lemma~\ref{lem:domino+1}, the vertices \(a_{5n+2}\), \ldots, \(a_{10n+1}\) are left of 
\(N(a_{5n}, b_{5n+1})\) and the vertices \(b_{5n+2}\), \ldots, \(b_{10n+1}\) are right of \(N(a_{5n}, b_{5n+1})\). See Figure~\ref{fig:separating-N-path-a}.
In this case, we may flip the drawing and swap each \(a_{i}\) and \(b_i\) with \(a_{10n+2-i}\) and
\(b_{10n+2-i}\) respectively, so that \(a_{1}\), \ldots, \(a_{5n}\) are right of
\(N(a_{5n+2}, b_{5n+1})\) and \(b_1\), \ldots, \(b_{5n}\) are right of \(N(a_{5n+2}, b_{5n+1})\).

  \begin{figure}
     \centering
     \begin{subfigure}[b]{0.45\textwidth}
         \centering
         \includegraphics[width=\textwidth]{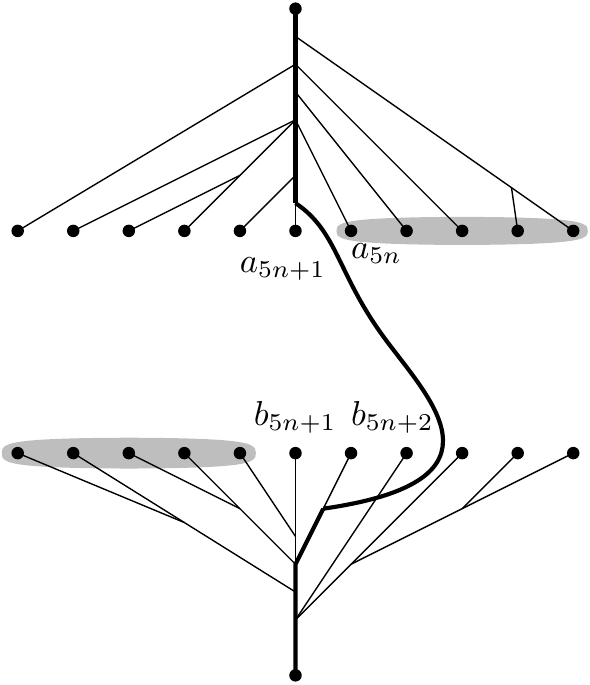}
         \caption{\(N(a_{5n+1}, b_{5n+2})\) separates
         \(a_1\), \ldots, \(a_{5n}\) from \(b_1\), \ldots, \(b_{5n}\).}
         \label{fig:separating-N-path-b}
     \end{subfigure}
     \hfill
     \begin{subfigure}[b]{0.45\textwidth}
         \centering
         \includegraphics[width=\textwidth]{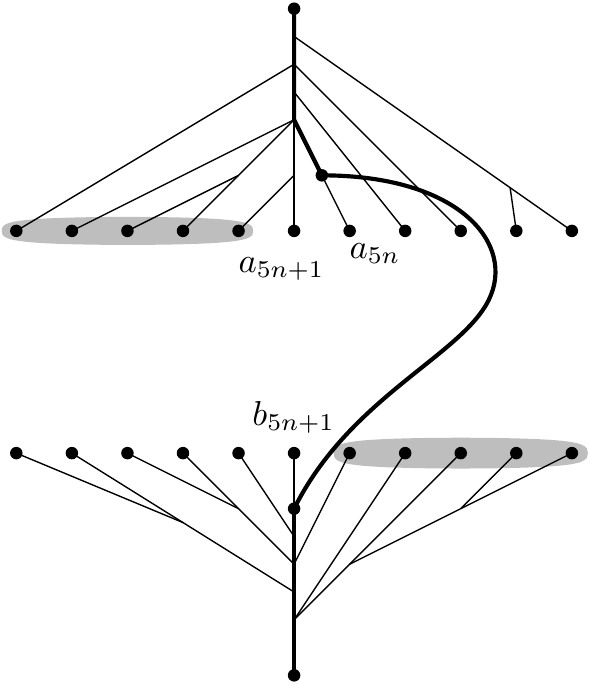}
         \caption{\(N(a_{5n}, b_{5n+1})\) separates
         \(a_{5n+2}\), \ldots, \(a_{10n+1}\) from \(b_{5n+2}\), \ldots, \(b_{10n+1}\).}
         \label{fig:separating-N-path-a}
     \end{subfigure}
      \caption{Lemmas~\ref{lem:domino},~\ref{lem:domino+1}, and~\ref{lem:separating-N-path} imply existence of one of these configurations.
      }
      \label{fig:separating-N-path}
  \end{figure}

Therefore, after possibly flipping the drawing and reversing the pairs in the standard example,
we assume that there exist \(a^* \in \{a_{5n+1}, b_{5n+2}\}\) and \(b^* \in \{b_{5n+1}, b_{5n+2}\}\)
with \(a^* \le b^*\) in \(P\) such that
\[
a_1, \ldots, a_{5n} \textrm{ are right of }N(a^*, b^*)
\quad\textrm{and}\quad
b_1, \ldots, b_{5n}\textrm{ are left of }N(a^*, b^*).
\]
Let us fix such \(a^*\) and \(b^*\), let \(N^* = N(a^*, b^*)\),
\(u^* = u(a^*, b^*)\) and \(v^* = v(a^*, b^*)\).

For each \(i \in [5n-1]\), let \(W_i'\) be a witnessing path from \(a_i\)
to \(b_{i+1}\) which minimises the number of edges outside \(N^*\).
This way, if \(W_i'\) has a nonempty intersection with any of the
witnessing paths \(x_0 N^* u^*\), \(v^* N^* u^*\) or \(v^* N^* y_0\),
then that intersection is a path.
It turns out that the intersection of \(W_i'\) with \(v^* N^* y_0\) is
always empty.

\begin{nclaim}\label{clm:W_i'-disjoint}
  For every \(i \in [5n-1]\), the witnessing path \(W_i'\) is disjoint from
  \(v^* N^* y_0\).
\end{nclaim}
\begin{subproof}
  Suppose to the contrary that \(W_i'\) intersects \(v^* N^* y_0\).
  Then \(W_i'\) is a witnessing path intersecting both \(a_i S y_0\)
  and \(a^* S y_0\). By Lemma~\ref{lem:three-paths}, \(W_i'\)
  intersects \(a_{i+1} S y_0\) as well. This implies
  \(a_{i+1} \le b_{i+1}\) in \(P\), a contradiction.
\end{subproof}

We know what it means for a vertex  of \(G\) to be left or right of \(N^*\).
We extend this definition to edges: an egde \(e = uv\) of \(G\)
is said to be \emph{left} (\emph{right}) of \(N^*\) if either it has an end which is
left (right) of \(N^*\), or its ends are non-consecutive vertices on \(N^*\)
and the paths \(u N^* y_0\), \(u N^* x_0\), and the edge \(e\) leave \(u\) in that clockwise (anti-clockwise) order.
(The definition does not depend on which end we take as \(u\)).

For every \(i \in [5n-1]\), \(a_i\) is right of \(N^*\) and \(b_{i+1}\) is
left of \(N^*\), so the path \(W_i'\) intersects \(x_0 N^* u^*\)
or \(v^* N^* u^*\) (or both).
Hence, the intersection of \(W_i'\) with \(N^*\) is either a subpath of \(x_0 N^* u^*\) or \(u^* N^* v^*\),
or the union of two disjoint subpaths of \(x_0 N^* u^*\) and \(u^* N^* v^*\) respectively.
In the former case, there exist vertices \(w_1\) and \(w_2\) on \(W_i'\)
with \(a_i < w_1 \le w_2 < b_{i+1}\) in \(P\) such that
\(W_i' \cap N^* = w_1 N^* w_2\), all edges of \(a_i W_i' w_1\) are right of
\(N^*\) and all edges of \(w_2 W_i' b_{i+1}\) are left of \(N^*\).
In the latter case, there exist vertices \(w_1\), \(w_2\), \(w_3\), \(w_4\) on \(W_i'\)
with \(a_i < w_1 \le w_2 < w_3 \le w_4 < b_{i+1}\) in \(P\) such that
\(W_i' \cap N^* = w_1 N^* w_2 \cup w_3 N^* w_4\), all edges of \(a_i W_i' w_1\) are right of
\(N^*\), all edges of \(w_4 W_i' b_{i+1}\) are left of \(N^*\), and either all edges
of \(w_2 W_i' w_3\) are left of \(N^*\), or all of them are right of \(N^*\).
Hence for every \(i \in [5n-1]\) there must exist a vertex \(w\) on the intersection of \(W_i'\) with \(N^*\)
such that no edge of \(a_i W_i' w\) is left of \(N^*\) and no edge of
\(w W_i' b_{i+1}\) is right of \(N^*\).
If there exist such \(w\) on \(x_0 N^* u^*\), then we say that
\(W_i'\) \emph{crosses} \(x_0 N^* u^*\) and similarly, if such \(w\)
can be found on \(u^* N^* v^*\), then we say that \(W_i'\) \emph{crosses} \(u^* N^* v^*\).
See Figure~\ref{fig:crossing}.
\begin{figure}
     \centering
     \begin{subfigure}[b]{0.19\textwidth}
         \centering
         \includegraphics{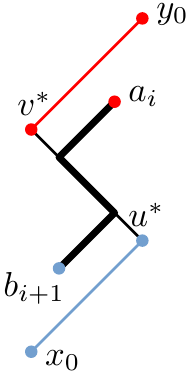}
         \caption{}
         \label{fig:cross-u-v-1}
     \end{subfigure}
     \hfill
     \begin{subfigure}[b]{0.19\textwidth}
         \centering
         \includegraphics{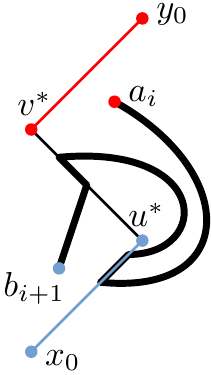}
         \caption{}
         \label{fig:cross-u-v-2}
     \end{subfigure}
     \hfill
     \begin{subfigure}[b]{0.19\textwidth}
         \centering
         \includegraphics{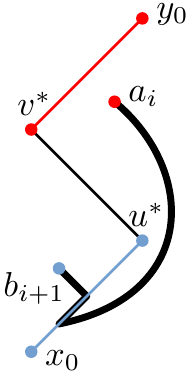}
         \caption{}
         \label{fig:cross-x-u-1}
     \end{subfigure}
     \hfill
     \begin{subfigure}[b]{0.19\textwidth}
         \centering
         \includegraphics{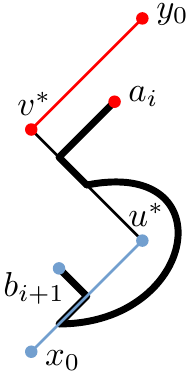}
         \caption{}
         \label{fig:cross-x-u-2}
     \end{subfigure}
     \begin{subfigure}[b]{0.19\textwidth}
         \centering
         \includegraphics{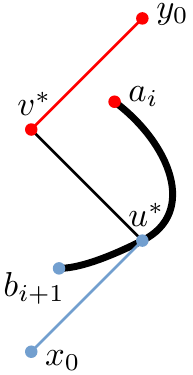}
         \caption{}
         \label{fig:cross-both}
     \end{subfigure}
        \caption{Ways in which \(W_i'\) can cross \(N^*\).
        In~\ref{fig:cross-u-v-1} and~\ref{fig:cross-u-v-2}, \(W_i'\)
        crosses \(u^* N^* v^*\).
        In~\ref{fig:cross-x-u-1} and~\ref{fig:cross-x-u-2}, \(W_i'\)
        crosses \(x_0 N^* u^*\).
        In~\ref{fig:cross-both}, \(W_i'\) crosses both \(u^* N^* v^*\) and
        \(x_0 N^* u^*\).
        }
        \label{fig:crossing}
\end{figure}

In the next part of the proof, we investigate, which paths among
\(W_1'\), \ldots, \(W_{5n-1}'\) cross \(u^* N^* v^*\).
We show, that 
the set of indices \(i \in [5n-2]\) such that the path \(W_i'\) crosses \(u^* N^* v^*\)
forms a range of consecutive integers.
Note that we only consider the paths with \(i \le 5n-2\) and the last path \(W'_{5n-1}\)
is not considered.
See Figure~\ref{fig:pretty} for an illustration.

\begin{figure}
    \centering
    \includegraphics{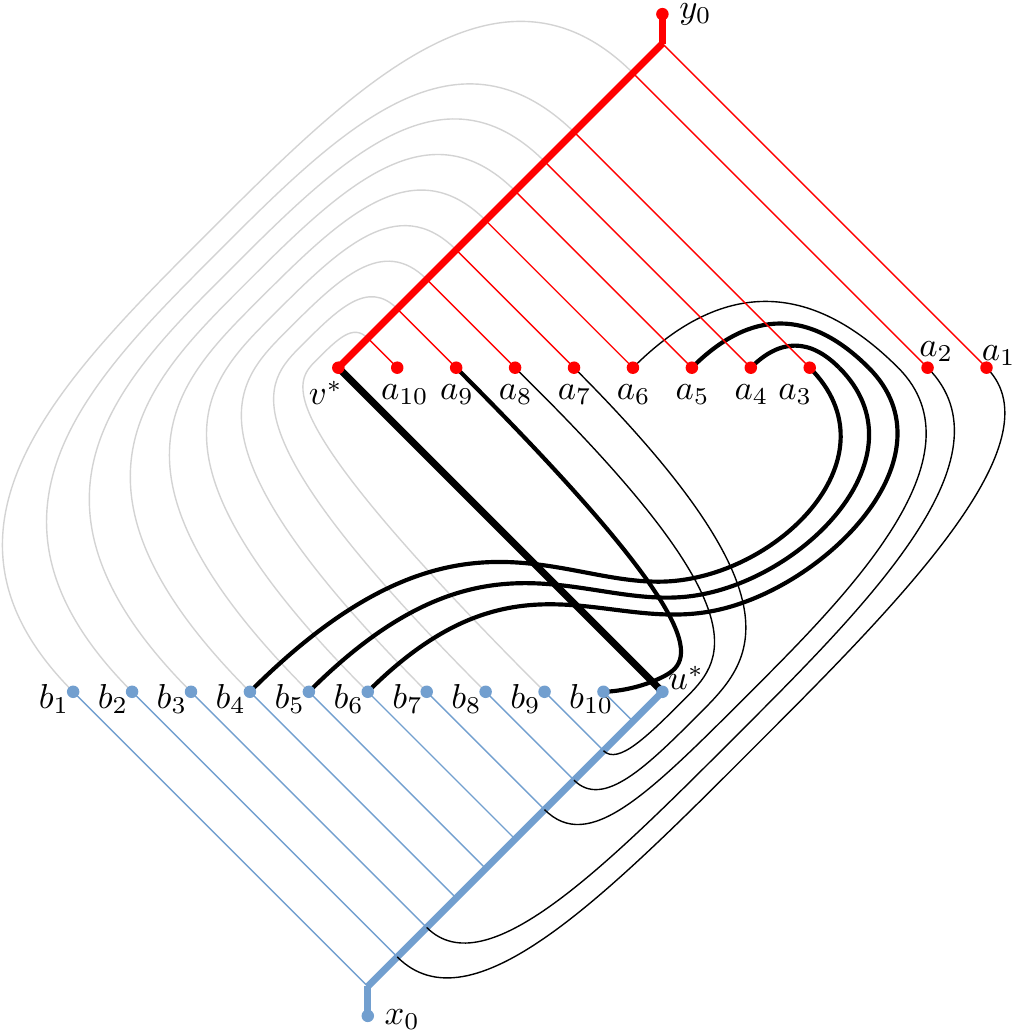}
    \caption{A standard example of size \(10\). The witnessing paths
    \(W_3'\), \(W_4'\), \(W_5'\), \(W_9'\) cross the path \(u^* N^* v^*\).
    After discarding the path \(W_9'\), we are left with the paths
    \(W_3'\), \(W_4'\), \(W_5'\) which have consecutive indices.
    }
    \label{fig:pretty}
\end{figure}

Suppose to the contrary that this is not true.
Then there exist non-consecutive integers \(i\) and \(k\) with
\(1 \le i < k \le 5n-2\) such that
\(W_i'\) and \(W_k'\) cross \(u^* N^* v^*\) and for every
\(j\) with \(i < j < k\), the path \(W_j'\) does not cross \(u^* N^* v^*\)
(and thus crosses \(x_0 N^* u^*\)).
Let us fix such \(i\) and \(k\).

Let \(w\) denote the least vertex of \(W_i'\)
which lies on \(u^* N^* v^*\).
Since \(W_i'\) crosses \(u^* N^* v^*\),
no vertex of the path
\(a_i' W_i' w\) is left of \(N^*\).
Let \(v\) denote the greatest vertex of \(W_i'\)
which lies on \(a_{i} S y_0\).
As \(w\) and \(v\) both lie on \(W_i'\) they are
comparable in \(P\).

\begin{nclaim}\label{clm:v-lt-w}
  We have \(v < w\) in \(P\).
\end{nclaim}
\begin{subproof}
Suppose towards a contradiction that we have \(w \le v\) in \(P\).
Then \(v^* N^* w W_i v\) is a witnessing path in \(P\)
which intersects both \(a^* S y_0\) and
\(a_i S y_0\).
Hence, by Lemma~\ref{lem:three-paths},
the path \(u^* N^* w W_i v\) intersects \(a_{i+1} S y_0\),
which implies that \(a_{i+1} \le v \le b_{i+1}\) in \(P\),
a contradiction.
\end{subproof}

\begin{nclaim}\label{clm:a_i'Sy_0}
  The path \(a_i S y_0\) is disjoint from \(u^* N^* v^*\).
\end{nclaim}
\begin{subproof}
Suppose to the contrary that \(a_i S y_0\) does intersect
\(u^* N^* v^*\) and let \(z\) denote the greatest vertex of that
intersection so that \(z S y_0\) is internally disjoint from
\(u^* N^* v^*\).
It is impossible that \(z\) lies on \(w N^* v^*\):
If it was the case, then the witnessing path \(w N^* v^*\)
would intersect both \(a_i S y_0\) and \(a^* S y_0\),
so by Lemma~\ref{lem:three-paths},
the witnessing path \(w N^* v^*\) would intersect \(a_{i+1} S y_0\)
and we would have \(a_{i+1} \le w \le b_{i+1}\) in \(P\).
Hence \(z\) does not lie on \(w N^* v^*\), which implies that
\(w\) is an internal vertex of \(z N^* v^*\) and in particular
\(w < z\) in \(P\).
Since \(a_i \prec_S a^*\) and \(a_i \le w < z\) in \(P\), the path
\(a_i S w\) is disjoint from the path \(v^* S z\).
Moreover, \(w\) is the only vertex of \(a_i S w\) which lies on
\(z N^* v^*\).
As \(a_i\) is right of \(N^*\), 
the vertex \(a_i\) must lie in the region
bounded by the cycle \(C = v^* S z N^* v^*\).
But \(C\) is contained in the union of the witnessing paths
\(v^* N^* y_0\) and \(v^* N^* z S y_0\), so
\(V(C) \subseteq \Up_P(v^*) \subseteq \Up_P(a^*)\).
Hence, \(a_i\) is enclosed by \(a^*\) and we obtain a contradiction
with Lemma~\ref{lem:enclosed}.
\end{subproof}

\begin{nclaim}\label{clm:two-internally-disjoint-paths}
  The paths \(v^* S v\) and \(v^* N^* w W_i' v\) are internally disjoint.
\end{nclaim}
\begin{subproof}
By definition of \(v\), the path \(w W_i' v\) intersects
\(v S y_0\) only in \(v\), and by Claim~\ref{clm:W_i'-disjoint},
\(w W_i' v\) is disjoint from \(v^* S y_0\).
Hence \(w W_i' v\) intersects \(v^* S v\) only in \(v\).
Since \(v^* S y_0 = v^* N^* y_0\), the path \(v^* N^* w\)
intersects \(v^* S y_0\) only in \(v^*\).
It remains to show that \(v^* N^* w\) is disjoint from
\(v S y_0\). Suppose to the contrary that
\(v S y_0\) intersects \(v^* N^* w\).
In such case, the witnessing path \(v^* N^* w\)
intersects both \(a^* S y_0\) and \(a_{i} S y_0\),
so by Lemma~\ref{lem:three-paths}, \(v^* N^* w\)
intersects \(a_{i+1} S y_0\) as well.
Thus, \(a_{i+1} \le w \le b_{i+1}\) holds in
\(P\), a contradiction.
\end{subproof}

By Claim~\ref{clm:two-internally-disjoint-paths}, the union
of the paths \(v^* S v\) and \(v^* N^* w W_i v\) is a cycle.
Let us denote that cycle by \(C\).
The intersection of \(C\) with \(N^*\) is a path,
and every vertex in the the region bounded by \(C\)
which does not lie on \(N^*\)
is right of \(N^*\). See Figure~\ref{fig:cp}.
\begin{figure}
    \centering
    \includegraphics{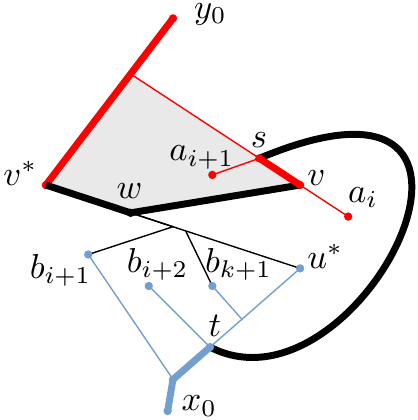}
    \caption{The shaded area above \( w \) is the region bounded by \(C\).
    The bolded \(x_0\)--\(y_0\) path is \(N_0\).}
    \label{fig:cp}
\end{figure}
We note that there does not have to exist \(j\) such that
\(V(C) \subseteq \Up_Q(a_j)\), so the following claim is not a contradiction
with Lemma~\ref{lem:enclosed}.

\begin{nclaim}\label{clm:a_j-inside}
  For each integer \(j\) with \(i + 1 \le j \le 5n\), the vertex
  \(a_j\) lies in the region bounded by \(C\).
\end{nclaim}
\begin{subproof}
Fix an integer \(j\) with \(i + 1 \le j \le 5n\),
and suppose towards a contradiction that \(a_j\)
does not lie in the region bounded by \(C\).
Let \(w'\) denote the least vertex of \(a_j S y_0\) which lies on
\(C\) so that all vertices of \(a_j S w'\) except \(w'\)
are outside the region bounded by \(C\).
Suppose that \(w'\) lies on \(w W_i' v\).
Then \(w W_i' v\) is a witnessing path intersecting both
\(a_i S y_0\) and \(a_j S y_0\), so
by Lemma~\ref{lem:three-paths},
the path \(w W_i' v\) intersects \(a_{i+1} S y_0\) as well,
and we have \(a_{i+1} \le w \le b_{i+1}\) in \(P\),
which is a contradiction.
Furthermore, since \(a_i \prec_S a_{j} \prec_S a^*\),
it is impossible that \(w'\) is an internal vertex of
\(v^* S v\).
Hence, \(w'\) has to be a vertex of \(w N^* v^*\)
distinct from \(w\) and thus we have \(w' < w\) in \(P\).
Since \(a_j\) is right of \(N^*\),
the path \(a_j S w'\) has to intersect \(x_0 N^* w\).
As the standard example is separated, the path
\(a_j S w'\) does not intersect \(x_0 N^* u^*\),
so it has to intersect \(u^* N^* w\).
But this is also impossible as \(V(a_j S w') \subseteq \Down_P(w')\),
\(V(u^* N^* w) \subseteq \Up_P(w)\) and \(w' < w\) in \(P\).
This contradiction proves the claim.
\end{subproof}

\begin{nclaim}\label{clm:wle}
  We have \(w \le b_{k+1}\) in \(P\).
\end{nclaim}
\begin{subproof}
Let \(w'\) denote the least element of the intersection
of \(W_k'\) with \(u^* N^* v^*\). By our assumption,
\(W_k'\) crosses \(u^* N^* v^*\), so \(w'\) is well defined,
no vertex on \(x_0 N^* w'\) is right of \(N^*\), and no
vertex on \(w' N^* y_0\) is left of \(N^*\).
Since both \(w\) and \(w'\) lie on the witnessing path
\(u^* N^* v^*\), they are comparable in \(P\).
If we have \(w \le w'\) in \(P\), then \(w \le w' \le b_{k+1}\)
holds in \(P\), so the claim holds.
Let us hence assume that \(w' < w\) in \(P\) and thus
\(w'\) lies on \(w N^* u^*\) and is distinct from \(w\).

The path \(w' W_k' b_{k+1}\) has to be disjoint from \(x_0 T b_{i+1}\)
because otherwise \(w' W_k' b_{k+1}\) would be a witnessing path intersecting
both \(x_0 T b_{i+1}\) and \(x_0 T b_{k+1}\), so by Lemma~\ref{lem:three-paths},
the path \(w' W_k' b_{k+1}\) would intersect \(x_0 T b_k\) as well and we would
have \(a_k \le w' \le b_k\) in \(P\).
Let \(u\) denote the least vertex on \(w' W_k' b_{k+1}\)
which lies on \(x_0 T b_{k+1}\).
Consider the path \(N = x_0 T u W_k' w' N^* y_0\).
Since  \(w' W_k' b_{k+1}\) is disjoint from \(x_0 T b_{i+1}\), we have
\(b_{i+1} \prec_T u\) and the vertex \(b_{i+1}\) must be left of \(N\).
Moreover, since \(w'\) lies on \(w N^* v^*\) and no vertex of \(N\) is right
of \(N^*\), the vertex \(w\) is not left of \(N\). Hence the path
\(w W_i' b_{i+1}\) has to intersect \(N\).
By Claim~\ref{clm:W_i'-disjoint}, it does not intersect \(v^* N y_0\),
and it does not intersect \(v^* N w'\),
since \(V(v^* N w') \subseteq \Down_P(w')\) and \(w' < w\) in \(P\).
Therefore \(w W_i' b_{i+1}\) has to intersect \(x_0 N w'\).
But \(V(x_0 N w') \subseteq \Down_P(u) \subseteq \Down_P(b_{k+1})\),
so we have \(w \le b_{k+1}\) in \(P\), as claimed.
\end{subproof}

By our assumption, the path \(W_{i+1}'\) does not cross the path \(u^* N^* v^*\),
so it has to intersect the path \(x_0 T u^*\).
Let \(t\) denote the least vertex of the intersection of \(W_{i+1}'\) with
\(x_0 T u^*\).

\begin{nclaim}\label{clm:tle}
  We have \(t \le b_{k+1}\) in \(P\).
\end{nclaim}
\begin{subproof}
  The witnessing path \(t W_{i+1}' b_{i+2}\) intersects
  \(x_0 T b_{i+2}\) and \(x_0 T b^*\), so by Lemma~\ref{lem:three-paths}
  it intersects \(x_0 T b_{k+1}\) as well. Hence we indeed have
  \(t \le b_{k+1}\) in \(P\).
\end{subproof}

By Claim~\ref{clm:a_j-inside}, \(a_{i+1}\) lies in the region
bounded by \(C\).
The vertex \(t\) does not lie in the interior of the
region bounded by \(C\) since it lies on \(N^*\).
Hence the path \(a_{i+1} W_{i+1}' t\) has to intersect the cycle \(C\).
By Claim~\ref{clm:W_i'-disjoint}, it does not intersect \(u^* S y_0\).
Let \(s\) denote the greatest vertex of the intersection of
\(a_{i+1} W_{i+1}' t\) with \(C\).
By Claim~\ref{clm:W_i'-disjoint}, \(s\) does not lie on \(u^* S y_0\),
and it does not lie on \(v^* N^* w W_i' v\), as that would imply
\(a_{i+1} \le s \le w \le b_{i+1}\) in \(P\). Hence,
in \(S\) the vertex \(s\) is ancestor of \(u\) but not of \(u^*\).
Consider the path
\[
  N_0 = x_0 T t W_{i+1}'s S v W_i' w N^* v^* T y_0.
\]
See Figure~\ref{fig:cp}.
By Claim~\ref{clm:a_j-inside}, the vertex \(a_{k+1}\) lies in the region bounded
by \(C\) and hence is not left of \(N_0\).
The vertex \(b_{k+2}\) is left of \(N^*\) and thus is left of \(N_0\).
Hence the path \(W_{k+1}'\) has to intersect \(N_0\).
By Claim~\ref{clm:W_i'-disjoint} is does not intersect \(v^* N_0 y_0\),
so it has to intersect \(x_0 N_0 v^*\).
But we have
\[
  V(x_0 N_0 v^*) = V(x_0 N_0 v) \cup V(v N_0 v^*) \subseteq \Down_P(t) \cup \Down_P(w),
\]
so by Claims~\ref{clm:tle} and~\ref{clm:wle} we have \(V(x_0 N_0 v^*) \subseteq \Down_P(b_{k+1})\).
Hence the fact that \(W_{k+1}'\) intersects \(x_0 N_0 v^*\) implies that
\(a_{k+1} \le b_{k+1}\) in \(P\).
This contradiction completes the proof that the indices \(i \in [5n-2]\) which cross \(u^* N^* v^*\)
form a range of consecutive integers.

Therefore, there exist integers \(i_1\) and \(k_1\) with \(1 \le i_1 \le k_1 \le 5n-1\),
such that for every \(j \in [5n-2]\), the path \(W_j'\) crosses \(u^* N^* v^*\)
if and only if \(i_1 \le j < k_1\) (if no path \(W_j'\) crosses \(u^* N^* v^*\), we can take
\(i_1 = k_1 = 1\)).
In particular, if \(i_1 \le j < k_1\) holds, then \(W_j'\) intersects \(u^* N^* v^*\),
and if \(i_1 \le j < k_1\) does not hold, then \(W_j'\) intersects \(x_0 T u^*\).
Fix such \(i_1\) and \(k_1\).

By a symmetric argument, there exists witnessing paths \(W_1''\), \ldots, \(W''_{5n-1}\)
and integers \(i_2\) and \(k_2\) with \(1 \le i_2 \le k_2 \le 5n-1\)
such that each \(W_j''\) is a witnessing path from \(a_{i+1}\) to \(b_i\),
and for every \(j \in [5n-2]\), if \(i_2 \le j < k_2\) holds, then \(W_j''\) intersects
\(u^* N^* v^*\) and if \(i_2 \le j < k_2\) does not hold, then \(W_j''\) intersects
\(v^* S y_0\). Let us fix such \(W_1''\), \ldots, \(W''_{5n-1}\), \(i_2\) and \(k_2\).

Removing the elements \(i_1\), \(k_1\), \(i_2\), and \(k_2\) from the set \([5n-1]\)
yields a set which is the union of at most \(5\) ranges of consecutive integers.
As the total number of elements in these ranges is at least \(5n-5\), one of these ranges
has size at least \(n - 1\).
Hence
there exists an integer \(j_0\) with \(1 \le j_0 < j_0 + n - 2 \le 5n-1\) such that 
the sequence \(j_0\), \ldots, \(j_0 + n - 2\), does not contain any of
\(i_1\), \(k_1\), \(i_2\), and \(k_2\).
Hence, there exist witnessing paths \(W_L \in \{v^* N^* u^*, v^* N^* y_0\}\) and
\(W_R \in \{v^* N^* u^*, x_0 N^* u^*\}\) such that
for every \(j \in \{j_0, \ldots, j_0 + n - 2\}\),
the path \(W_j'\) intersects \(W_R\) and the path \(W_j''\) intersects \(W_L\).

For each \(j \in [n]\), let \(a_j' = a_{j_0 + j - 1}\) and \(b_j' = b_{j_0 + j - 1}\),
let \(c_j'\) be any element of the intersection of \(W_{j_0 + j -1}'\) with \(W_R\),
and let \(d_j'\) denote any element of the intersection of \(W_{j_0 + j -1}'\) with \(W_L\).
Let \(K'_n\) denote the subposet of \(P\) induced by the elements
\(a_1'\), \ldots, \(a_n'\),  \(b_1'\), \ldots, \(b_n'\), \(c_2'\) \ldots, \(c_{n-2}'\), \(d_2'\) \ldots, \(d_{n-2}'\). 

The elements \(c_1'\), \ldots, \(c_{n-1}'\) are pairwise comparable in \(P\) because they
all lie on \(W_R\). For each \(j \in [n-2]\) we have \(c_j' < c_{j+1}\) in \(P\) as otherwise
we would have \( a_{j+1}' \le c_{j+1}' \le c_j' \le b_{j+1}'\). Hence, we have
\(c_1' < \cdots < c_{n-1}'\) in \(P\) and by a symmetric argument, we have
\(d_{n-1}' < \cdots < d_1'\) in \(P\).
Moreover, for each \(j \in [n-1]\) we have \(a_j' \le c_j' \le b_{j+1}'\) and
\(a_{j+1} \le d_j' \le b_j'\) in \(P\).
Hence, for any pair of elements \(x\) and \(y\) in \(\Kelly{n}\), if \(x \le y\) in \(\Kelly{n}\),
then for the corresponding pair \(x'\) and \(y'\) of \(\Kelly{n}'\) we have \(x' \le y'\) in \(\Kelly{n}'\).

Furthermore, for every pair \((x, y) \in \Inc(\Kelly{n})\), \(x'\) must be incomparable with \(y'\).
Suppose to the contrary that \(y' \le x'\) in \(\Kelly{n}'\).
There exists a pair \((a_j, b_j) \in \Inc(\Min(\Kelly{n}), \Max(\Kelly{n}))\) such that
\(a_j \le y\) and \(a_j \not \le x\) in \(\Kelly{n}\). By definition of \(\Kelly{n}\), we necessarily have
\(x \le b_j\) in \(\Kelly{n}\). Hence we have \(a_j' \le y' \le x' \le b_j'\) in \(\Kelly{n}'\), a contradiction.
We can exclude the possibility that \(x' \le y'\) in \(\Kelly{n}'\) with a symmetrical argument.
This proves that \(\Kelly{n}'\) isomorphic to the Kelly poset \(\Kelly{n}\).

Finally, since the standard example \(\{(a_1', b_1'), \ldots, (a_n', b_n')\}\) is doubly
exposed by \((x_0, y_0)\), it is impossible that \(\Kelly{n}'\) contains \(x_0\) or \(y_0\).

The proof of Lemma~\ref{lem:se-vs-kelly} is complete.
\end{proof}

\section{Large Kelly subposets and outerplanarity of cover graphs}\label{sec:kelly-and-k-outerplanarity}
%\section{Kelly subposets in posets with \(k\)-outerplanar cover graphs}\label{sec:kelly-and-k-outerplanarity}

%\section{Kelly subposets enforce large \( k \)-outerplanarity}\label{sec:kelly-and-k-outerplanarity}

The proof of Lemma~\ref{lem:outerplanarity-of-kelly} consists of two main parts.
First we show that if a poset \( P \) has a Kelly subposet, then its cover graph contains the cover graph of the Kelly subposet as a minor.
Secondly we argue that having the cover graph of a sufficiently large Kelly poset as a minor prevents a planar
graph from being \( k \)-outerplanar.

Recall that the elements \(a_1\), \ldots, \(a_n\), \(b_1\), \ldots, \(b_n\),
\(c_2\), \ldots, \(c_{n-2}\), \(d_2\), \ldots, \(d_{n-2}\)
of \(\Kelly{n}\) were defined by
\[
a_i = \{i\},\quad b_i = [n]\setminus\{i\},\quad c_i = \{1, \ldots, i\},\quad d_i = \{i+1, \ldots, n\},
\]
and thus \(c_1= a_1\), \(c_{n-1} = b_n\), \(d_1 = b_1\) and
\(d_{n-1} = a_n\).
%We will denote that cover graph of \( \Kelly{n} \) by \( \cover(\Kelly{n}) \) \ComMa{to spite Diestel}.

\begin{lemma}\label{lem:kelly-implies-minor}
  If \(n \ge 2\) and a poset \( P \) contains a subposet isomorphic to the Kelly poset \(\Kelly{n}\),
  then \(\cover(\Kelly{n})\) is a minor of \(\cover(P)\).
\end{lemma}
\begin{proof}
  After renaming the elements of \(P\), we assume that
  not only does \(P\) have a subposet isomorphic to
  \(\Kelly{n}\), but actually \(\Kelly{n}\) is a subposet of \(P\).
  We shall construct a \emph{minor model} of \(\cover(\Kelly{n})\), that is a family
  \(\{\phi(x) : x \in \Kelly{n}\}\) of pairwise disjoint connected sets of vertices in
  \(\cover(P)\) such that
  that for every edge \(x y \in \cover(\Kelly{n})\) there
  exists an edge between \(\phi(x)\) and \(\phi(y)\) in \(\cover(P)\).
  The existence of such a model will prove the lemma because
  after contracting each set \(\phi(x)\) to the vertex \(x\)
  in \(\cover(P)\) we obtain a graph containing \(\cover(\Kelly{n})\) as a subgraph.
  
  Let \(W_c\)
  be a witnessing path from \(c_1\) to \(c_{n-1}\) in \(P\)
  which contains all elements
  \(c_1\), \ldots, \(c_{n-1}\), 
  and let
  \(W_d\) be a witnessing path from \(d_{n-1}\) to \(d_{1}\) in \(P\)
  which contains all elements \(d_{n-1}\), \ldots, \(d_{1}\).
  Observe that no vertex of \(W_c\) is comparable with a vertex
  with \(W_d\) in \(P\) as that would imply either
  \(a_1 = c_1 \le d_1 = b_1\) or \(a_n = d_{n-1} \le c_{n-1} = b_n\)
  in \(P\).
  For each pair of elements \(x\) and \(y\) of \(\Kelly{n}\), let
  \(W(x, y)\) denote a witnessing path from \(x\) to \(y\) in \(P\)
  which minimises the number of edges outside \(W_c \cup W_d\).
  As there are no comparabilities between \(V(W_c)\) and \(V(W_d)\),
  if \(W(x, y)\) intersects \(W_c \cup W_d\), then that intersection is
  a subpath of \(W_c\) or \(W_d\).
  We can now define our minor model by
  \begin{align*}
    \phi(a_i) &= (V(W(a_i, c_i)) \setminus V(W_c)) \cup
    (V(W(a_i, d_{i-1}))) \setminus V(W_d)) && \textrm{for }2 \le i \le n-1\\
    \phi(b_i) &= (V(W(c_{i-1}, b_i)) \setminus V(W_c)) \cup
    (V(W(d_i, b_{i}))) \setminus V(W_d)) && \textrm{for }2 \le i \le n-1\\
    \phi(c_i) &= (V(W_c) \cap \Up_P(a_i)) \setminus \Up_P(a_{i+1}) && \textrm{for }1 \le i \le n-1\\
    \phi(d_i) &= (V(W_c) \cap \Down_P(b_i)) \setminus \Down_P(b_{i+1}) && \textrm{for }1 \le i \le n-1.
  \end{align*}
  Note that since \(a_1 = c_1\), \(a_n = d_{n-1}\), \(b_1 = d_{n-1}\), and \(b_n = c_{n-1}\),
  \(\phi(x)\) is defined for all \(x \in \Kelly{n}\).
  It is easy to see that the sets  \(\phi(x)\) are connected in \(\cover(P)\).
  It is a tedious, yet easy task to rigorously verify that the sets \(\phi(x)\) are disjoint and
  form a minor model of \(\cover(\Kelly{n})\) (otherwise we would have \(a_i \le b_i\) in \(P\) for some
  \(i \in [n]\)).
\end{proof}

\begin{lemma}\label{lem:kelly-implies-kouterplanar}
  For every positive integer \(k\), the cover graph of \(\Kelly{4k+3}\)
  is not \(k\)-outerplanar.
\end{lemma}
\begin{proof}
  Let \(G = \cover(\Kelly{4k+3})\) and
  fix a planar drawing of \(G\).
  For every \(i \in [2k+1]\), we denote by \(C_i\) the cycle in \(G\) defined by
  \[
    C_i = G[\{a_{2i}, c_{2i}, c_{2i-1}, b_{2i}, d_{2i}, d_{2i-1}\}].
  \]
  The cycles \(C_1\), \ldots, \(C_{2k+1}\)
  are pairwise disjoint.
  For every \(i \in \{2, \ldots, 2k\}\),
  the cycle \(C_{i-1}\) is adjacent to the vertices \(c_{2i-1}\) and
  \(d_{2i-1}\) of \(C_i\), and the cycle \(C_{i+1}\) is adjacent to the vertices
  \(c_{2i}\) and \(d_{2i}\) of \(C_{i}\).
  Since the vertices \(c_{2i}\), \(c_{2i-1}\), \(d_{2i}\), and \(d_{2i-1}\) lie on \(C_i\) in that
  cyclic order, we conclude that one of the cycles \(C_{i-1}\) and \(C_{i+1}\) has all its vertices
  in the interior of the region bounded by \(C_i\) and the other one has all its
  vertices outside the region bounded by \(C_i\).
  
  \begin{figure}[ht]
    \centering
    \includegraphics{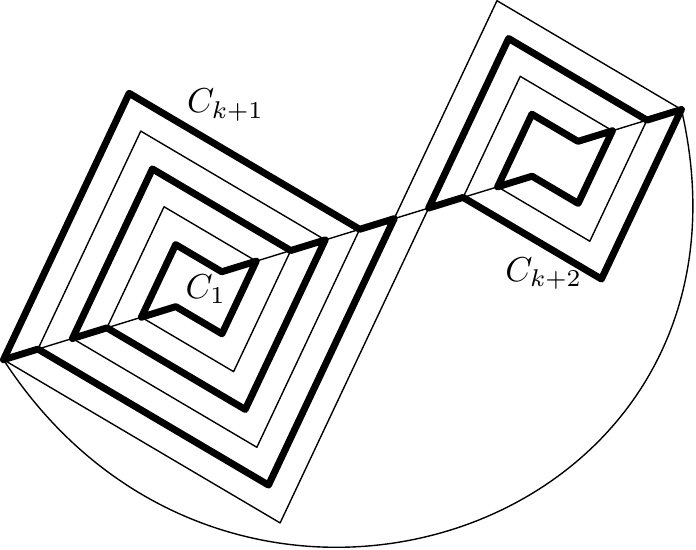}
    \caption{A planar drawing \(\cover(\Kelly{11})\). The bold cycles
    \( C_1 \), \(C_2 \), and \( C_3 \),
    are nested, witnessing that this drawing is not \(2\)-outerplanar.}
    \label{fig:kellycycles}
  \end{figure}
  
  Consider the cycle \(C_{k+1}\). 
  Without loss of generality we assume that
  the \(C_k\) has all its vertices in the region bounded by \(C_{k+1}\).
  A simple induction shows that for every \(i \in [k]\), \(C_i\) lies in the interior of the region bounded by
  \(C_{i+1}\).
  This implies that  the \(k\)-fold removal of vertices from the exterior face leaves
  the vertices of the cycle \(C_1\) intact.
  Since the planar drawing was chosen arbitrarily, we deduce that \(G\) is not \(k\)-outerplanar, as claimed.
\end{proof}

It is now easy to provide a proof of Lemma~\ref{lem:outerplanarity-of-kelly}.
\begin{proof}[Proof of Lemma~\ref{lem:outerplanarity-of-kelly}]
  Let \( G \) be the \( k \)-outerplanar cover graph of the poset \( P \).
  Due to Lemma~\ref{lem:kelly-implies-minor}, this implies that \( \kappa(P) \le 4k + 2 \).
  Otherwise \( \cover(K_{4k + 3}) \) is a minor of \( G \), contradicting Lemma~\ref{lem:kelly-implies-kouterplanar}.
\end{proof}

\bibliographystyle{plain}
\bibliography{main}

\end{document}